\documentclass[journal]{article}

\usepackage{graphicx,url}
\usepackage{dcolumn}
\usepackage{bm}
\usepackage{tikz}
\usetikzlibrary{matrix}
\usepackage{amsmath,amsfonts,amssymb}
\usepackage{pst-node}
\usepackage{tikz-cd}

\newtheorem{definition}{\noindent{\it Definition}}[section]
\newtheorem{theorem}{\noindent{\it Theorem}}[section]
\newtheorem{proposition}[theorem]{\noindent{\it Proposition}}

\newtheorem{remark}[theorem]{\noindent{\it Remark}}
\newtheorem{corollary}[theorem]{\noindent{\it Corollary}}
\newenvironment{proof}{\noindent{\it Proof:}}{$\hfill$ $\Box$\\ }
\newtheorem{example}{\noindent{\it Example}}[section]

\begin{document}

\title{On a categorical theory for emergence}

\author{Giuliano G. La Guardia, Pedro Jeferson Miranda
\thanks{Giuliano G. La Guardia (corresponding author) is with Department of
Mathematics and Statistics, State University of Ponta Grossa (UEPG),
84030-900, Ponta Grossa - Brazil, e-mail: (gguardia@uepg.br).
Pedro Jeferson Miranda is with Department of Physics,
State University of Ponta Grossa (UEPG), 84030-900, Ponta Grossa - Brazil,
e-mail: (pedrojemiranda@hotmail.com).}}

\maketitle

\begin{abstract}
It is well-known that biological phenomena are emergent. Emergent phenomena are 
quite interesting and amazing. However, they are difficult to be understood. 
Due to this difficulty, we propose a theory 
to describe emergence based on a powerful mathematical tool, namely,
Theory of Categories. In order to do this, we first utilize constructs
(categories whose objects are structured sets), their operations and
their corresponding generalized underlying functor (which are not necessary faithful)
to characterize emergence. After this, we introduce and show several results concerning 
homomorphism (isomorphism) between 
emergences, representability, pullback, pushout, equalizer, product and 
co-product of emergences among other concepts. Finally, we
explain how our theory fits in studies involving biological systems. 
\end{abstract}


\section{Introduction}

Scientific endeavor is based on the assumption that everything can be
analyzed and understood by the study of the parts and,
consequently, the understanding of the wholes. It is known that there are two
kinds of wholes based on how parts combine themselves in order to
produce effects. The first type of whole is the \emph{whole of order},
in which the behavior of parts are independent of the whole. In
particular, this type of whole is the product of an homeopathic
law in Mill's terminology, and corresponds to a combination of
the parts as causes in which their exact role is apprehensible
in the same whole \cite{mill:1974}. These kinds of phenomena
are observed in non-animated bodies, collective behavior,
Newtonian systems, and so on. The second type of whole is the
\emph{substantial whole}, in which the behavior of its parts is
determined by such whole. Analogously, using Mill's terminology,
this sort of whole is the product of an heteropathic law, that is,
the parts combine themselves as causes to generate an effect in
which the specific action of each part is lost in such combination
\cite{mill:1974}. In this case, the parts are dependent upon the
behavior of the whole.

Modern science has as premise that every phenomenon, conceived
as a whole, can be separated into smaller and simpler parts;
then, such parts can be studied separately in order to induce a global
understanding of the whole. The great success of this method,
that is, the \emph{reducionistic Cartesian Method}, is observed
in the Newtonian Physics and Classical Mechanics in general.
Furthermore, we can roughly admit that every non-living system,
from the atom to the galaxy, can be understood by Physics.
However, there are systems that are resilient to Cartesian
reductionism. A system can be considered reductive if there
is no loss of information when it is broken into parts.
Systems of this kind are often studied or modeled by means of
differential equations, linear or non-linear systems,
probability theory, and so on \cite{bertalanffy:1968}.
On the other hand, systems that are not reductive, i. e.,
resilient to the Cartesian method, cannot be broken
into parts without loss of information \cite{ehresmann:2007}.
It is important to note that the models mentioned above, which
can be applied to reductive systems, cannot be applied to
systems that are resilient to reduction.

The main type of resilient systems in terms of the Cartesian
reduction is the biological system. Biological systems
are capable of building themselves in the sense of \emph{autopoiesis}
presented by Maturana and Varela \cite{maturana:1980}.
Beside such concept, we roughly attribute life phenomena to
bear three sorts of behaviors: assimilation, growing and
reproduction. The assimilation is the phenomenon in which a system
absorbs energy and matter in order to constitute itself.
Secondly, the growing is the phenomenon in which a system
increases its size and movement; this includes healing,
regeneration, and asexual reproduction. Finally, the reproduction is
the phenomenon in which a system perpetuates itself over
time by a tension between heredity and variability. These
three phenomena are ubiquitous in living systems.

It is known that wholes of order are permeated by reducionistic
features, while substantial wholes are permeated by
emergent features. Roughly speaking, the emergent phenomenon is
understood by the concept resumed in the assertion ``the whole
is bigger than the sum of its parts". More precisely,
an emergent phenomenon is defined as a feature that occurs
in the whole and cannot be deduced from its parts. This means
that there is a correspondence between emergence and
systems which are resilient to the Cartesian method, that is,
non-reductive \cite{chalmers:2006}. More specifically, every
emergent system is non-reductive, while not every non-reductive
system is emergent. As an example of non-reductive systems which are
not emergent, let us consider Chalmers´ definition of weak
emergence \cite{chalmers:2006}: a property that is not expected
in the whole by the investigation of its parts, namely,
properties that arises from solids, liquids, and gases,
which are studied by Statistical Physics.

On the other hand, strong emergence stands for non-reductive systems
that are also emergent \cite{chalmers:2006}. All living systems are of
this kind since their main features are not expected in addition to being
not deducible from its parts. In such context, the aim of this paper
is to build a theory for emergent phenomena that occur mainly
in biological systems.

This paper is arranged as follows. In Section~\ref{sec2}, we review
the concept of emergence. Section~\ref{sec3} summarizes some
known results on Category Theory that will be utilized in this paper.
In Section~\ref{sec4}, we present the contributions of this paper,
\emph{i.e.}, a theory on emergence based on categories.
In Section~\ref{sec5}, we expose the Relational Biological
Theory of Robert Rosen, which stands for the empirical mean to
go from the concrete phenomenon to its corresponding category.
Section~\ref{sec6} shows some examples of application
of our theory and, in Section~\ref{sec7}, a brief discussion
on the proposed theory and its results is presented.

\section{Emergence}\label{sec2}

In this section we recall the concept of emergence as well as
its historical development. The first philosopher to give
a precise definition of emergence was George Henry Lewes~\cite{lewes:1875}
in his work about \emph{The problems of life and mind}:
``Thus, although it effect is the
resultant of its components, the product of its factors, we cannot
always trace the steps of the process, so as to see in the product
the mode of operation of it factor. In these latter case, I propose
to call defect an emergent". In other words, an emergence is a
phenomenon that cannot be understood by the investigation of
its components, \emph{i.e.}, its parts.

Investigating the literature of this concept, we will find
ourselves with the \emph{status quaestionis} of emergence.
Thus, we must consider the British Emergence School in the late
nineteen century that has as main authors: John Stuart Mill,
and Charlie Dunbar Broad.

The main contribution of Mill was to define two kinds of laws:
heteropathic and homopathic laws. The latter stands for
laws that cannot be explained by means of sums, products, associations,
superpositions, compositions, commutations or any other
theoretical device of the sort; that former is the opposite case.
In other words, something is ``forgotten in
the whole when parts are put together as causes" \cite{mill:1974}.

Following the same venue, Broad introduced the concept of trans-ordinal law which
corresponds to a hierarchical level that determines a phenomenon
in a lower level. This means that the whole has preference in the order of causes
\cite{broad:1925}.

These were the pioneers thinkers of the concept of emergence.
In the contemporary time, we consider Adrian Bejan, David J.
Chalmers, Nils A. Baas, and David Ellerman as the main researchers in
this area.

Bejan proposed in his works a new law of thermodynamics in order
to explain the way that systems are built. He named this law
as Constructal Law, which stands for an attempt to describe self-organizations
and, consequently, emergence~\cite{bejan:2016,bejan:2017}.

Chalmers sustained that there exists an empirical criterion
that classifies wholes by means of strong and weak
emergences. The weak emergence stands for the phenomenon by which
a whole has a property that arises from the parts
in a non expected way. On the other hand, strong emergence represents
a phenomenon that arises in the whole by a special relationships
between parts that is intrinsically
non-deductive~\cite{chalmers:2006}.

Baas introduced the utilization of Category Theory in order to
model hierarchies and hyperstructures. In fact,
he modeled emergence as a collective behavior \cite{baas:2006}.
This concept of emergence assumed by Baas is 
different from our approach. Additionally, Baas did
not intent to formalize and generate new results
by means of a strong definition of emergence.

Ellerman utilized, as Baas, Category Theory to work with
the concept of emergence. However, he did it by means of
what he defined as ``determination through universals" \cite{ellerman:2007},
while Baas utilized hyperstructures. His theory is developed upon
adjoint functors in order to describe ``heteromorphic" structures.
Besides the novelty of his approach, he did not formalize the
concept.

\section{Preliminaries in Category Theory}\label{sec3}

In this section, we review some basic concepts and results on Category
Theory necessary for the development of this work. This theory was
introduced by Eilenberg and Mac Lane (see \cite{maclane:1945}). For more details concerning
such theory we suggest the references~\cite{strecker:1990,maclane:1998}.

Recall that a \emph{category} is a quadruple $ \mathcal{A}=
( \operatorname{Ob}(\mathcal{A}), \operatorname{hom},
\emph{id}, \circ )$ consisting of\\
(1) a class of $\mathcal{A}$-objects denoted by $\operatorname{Ob}(\mathcal{A})$;\\
(2) for each pair of $\mathcal{A}$-objects $A$ and $B$, there is a set
$\operatorname{hom}(A, B)$ whose members are called $\mathcal{A}$-morphisms from $A$ to
$B$ and represented by $f: A \longrightarrow B$ (or $A \xrightarrow{f} B$;\\
(3) for each $\mathcal{A}$-object $A$ there exists an $\mathcal{A}$-morphism
${id}_{A}: A \longrightarrow A$ called the $\mathcal{A}$-identity on $A$;\\
(4) a composition law that associates each $\mathcal{A}$-morphisms
$f:A \longrightarrow B$ and $g:B \longrightarrow C$ to
an $\mathcal{A}$-morphism $g\circ f: A \longrightarrow C$ (composite of $f$ and $g$).
Such composition law is associative, preserves identity and the
$\operatorname{hom}$ sets are pairwise disjoints.

The class of all $\mathcal{A}$-morphisms is denoted by
$\operatorname{Mor}(\mathcal{A})=\bigcup \operatorname{hom}$.
If $f:A \longrightarrow B$ is an $\mathcal{A}$-morphism then $A$
is the domain and $B$ is the codomain of $f$.

In this paper, the category $\mathcal{S}$ whose class of objets consists of
all sets and the class $\operatorname{Mor}(\mathcal{S})$ of morphisms
consists of all functions is fundamental for the development of our theory.

Let $\mathcal{A}$ be a category. An $\mathcal{A}$-morphism
$f: A\longrightarrow B$ is called \emph{isomorphism} if
there exists an $\mathcal{A}$-morphism $g:B\longrightarrow A$ such that
$g \circ f = {id}_{A}$ and $f \circ g = {id}_{B}$; $g$ is said
to be the inverse of $f$.

Let $\mathcal{A}$ and $\mathcal{B}$ be two categories. Recall that a
\emph{functor} from $\mathcal{A}$ to $\mathcal{B}$ is a function
that assigns to each $\mathcal{A}$-object $A$ a $\mathcal{B}$-object $F(A)$,
and to each $\mathcal{A}$-morphism $A\xrightarrow{f}A^{'}$ a
$\mathcal{B}$-morphism $F(A)\xrightarrow{F(f)}F(A^{'})$ such
that $F$ preserves composition and identities morphisms.
Let $F:\mathcal{A} \longrightarrow\mathcal{B}$ be a functor.
We say that $F$ is \emph{embedding} if it is injective on morphisms.
$F$ is called \emph{faithful} if all hom-set restrictions
$F:{\operatorname{hom}}_{\mathcal{A}}(A, A^{*})\longrightarrow
{\operatorname{hom}}_{\mathcal{B}}(F(A), F(A^{*}))$ are injective.
Furthermore, $F$ is called \emph{full} if all hom-set restrictions
are surjective. If $A, B \in \operatorname{Ob}(\mathcal{A})$, then the
\emph{identity functor} ${id}_{\mathcal{A}}:\mathcal{A}
\longrightarrow \mathcal{A}$ is the functor given by ${id}_{\mathcal{A}}
(A\xrightarrow{f} B)= A\xrightarrow{f} B$. Given an
$\mathcal{A}$-object $A$, there exists the \emph{covariant hom-functor}
$\operatorname{hom}(A,-):\mathcal{A}\longrightarrow
\mathcal{S}$ such that
$$\operatorname{hom}(A,-)(B\xrightarrow{f} C)={\operatorname{hom}}(A, B)
\xrightarrow{\operatorname{hom}(A, f)}{\operatorname{hom}}(A, C),$$ where
$\operatorname{hom}(A, f)(g) = f \circ g$. Analogously, if
$A \in \operatorname{Ob}(\mathcal{A})$,
then there exists a \emph{contravariant hom-functor}
$\operatorname{hom}(-,A):{\mathcal{A}}^{op} \longrightarrow {\mathcal{S}}$ defined on
any ${\mathcal{A}}^{op}$-morphism $B\xrightarrow{f} C$ by
$$\operatorname{hom}(-,A)(B\xrightarrow{f} C)=
{\operatorname{hom}}_{\mathcal{A}}(B, A)\xrightarrow{\operatorname{hom}(f, A)}
{\operatorname{hom}}_{\mathcal{A}}(C, A),$$ with $\operatorname{hom}(f, A)
(g)= g \circ f$, where the composition is the same as in
$\mathcal{A}$.

Let $F: \mathcal{A}\longrightarrow\mathcal{B}$ and $G: \mathcal{B}
\longrightarrow \mathcal{C}$ be two functors, $A, A^{'}
\in \operatorname{Ob}(\mathcal{A})$ and $f \in \operatorname{hom}(A, A^{'})$.
Then the composite $G \circ F :\mathcal{A}
\longrightarrow\mathcal{C}$ defined by $(G \circ F)
(A\stackrel{f}{\longrightarrow}{A^{'}}) =G(FA)\xrightarrow{G(Ff)}G(F A^{'})$
is a functor.

A functor $F:\mathcal{A} \longrightarrow \mathcal{B}$ is an
\emph{isomorphism} if there exists a functor $G:\mathcal{B}
\longrightarrow \mathcal{A}$ such that $G\circ F = {id}_{\mathcal{A}}$
and $F\circ G = {id}_{\mathcal{B}}$, where ${id}_{\mathcal{A}}$
is the identity functor from $\mathcal{A}$ to $\mathcal{A}$.
In this case we say that the categories $\mathcal{A}$ and
$\mathcal{B}$ are isomorphic, denoted by $\mathcal{A}
\cong \mathcal{B}$.

Let $F, G: \mathcal{A}\longrightarrow\mathcal{B}$ be functors. A
\emph{natural transformation} $\tau$ from $F$ to $G$,
$F\xrightarrow{\tau}G$ is a function that assigns to each $\mathcal{A}$-object $A$
a $\mathcal{B}$-morphism ${\tau}_{A}:FA \longrightarrow GA$ in such a way
that, for each $\mathcal{A}$-morphism $f: A\longrightarrow A^{*}$, the diagram
\begin{center}
\begin{tikzpicture}
  \matrix (m) [matrix of math nodes,row sep=3em,column sep=3em,minimum width=2em]
  {
     FA & GA \\
     FA^{*} & GA^{*} \\};
  \path[-stealth]
    (m-1-1) edge node [left] {$Ff$} (m-2-1)
            edge node [above] {${\tau}_{A}$} (m-1-2)
    (m-2-1) edge node [below] {${\tau}_{A^{*}}$} (m-2-2)
    (m-1-2) edge node [right] {$Gf$} (m-2-2);
\end{tikzpicture}
\end{center}
commutes. A \emph{natural
transformation} $F\xrightarrow{\tau}G$ is called a \emph{natural
isomorphism} if ${\tau}_{A}:FA \longrightarrow GA$ is a
$\mathcal{B}$-isomorphism for each $\mathcal{A}$-object. A functor
$F: \mathcal{A} \longrightarrow \mathcal{S}$ is said to be
\emph{representable} (by an $\mathcal{A}$-object $A$) if $F$
is naturally isomorphic to the $\operatorname{hom}$-functor
$\operatorname{hom}(A, -):\mathcal{A} \longrightarrow \mathcal{S}$.

A \emph{quasi-category} is a quadruple $\mathcal{A}=( \mathcal{O},
\operatorname{hom}, id, \circ )$ such that the following conditions hold:\\
(1) $\mathcal{O}$ is a conglomerate, the members of which are called objects;\\
(2) for each pair $(A, B)$ of objects, $\operatorname{hom}(A, B)$ is a
conglomerate, called the conglomerate of all morphisms from $A$ to $B$;\\
(3) for each object $A$, ${id}_{A}: A \longrightarrow A$ is called the
identity morphism on $A$;\\
(4) for each pair of morphisms $f: A \longrightarrow B$,
$g: B \longrightarrow C$, there exists a composite morphism
$g \circ f: A \longrightarrow C$ that satisfies the associative property,
preserves identities, and all pairs of $\operatorname{hom}$ are pairwise disjoint.

\section{Modeling emergence by means of Category Theory}\label{sec4}

In this section, we present the contributions of this work, \emph{i.e.},
we propose a theory for emergent phenomena based on Category Theory.

Let $A$ be a set. An \emph{internal operation} on $A$ is a function
$f:A\times A\longrightarrow A$. If $K$ is a ring or a field,
we say that a function $g:K\times A \longrightarrow A$ is an \emph{external
operation} on $A$. An \emph{operation} on $A$
is an internal operation on $A$ or an external operation on $A$.

In Definition~\ref{defnew8}, we generalize the concept of construct
in the sense that we allow the existence of o finite set of
operations on the objects of a given category.

\begin{definition}\label{defnew8}
A structure in a category $\mathcal{A}$ is a finite set
of operations on $A$ $e_{\mathcal{A}}=
\{e_{\mathcal{A}}^{(1)}, e_{\mathcal{A}}^{(2)}, \ldots , e_{\mathcal{A}}^{(n)}\}$,
such that each $\mathcal{A}$-object $A$ has
$e_{A}= \{e_{A}^{(1)}, e_{A}^{(2)}, \ldots , e_{A}^{(n)}\}$ operations such
that $e_{\mathcal{A}}^{(i)}$ and $e_{A}^{(i)}$ have the same properties,
for every \ $i=1, 2, \ldots , n$. If $\mathcal{A}$ has a structure, then it is
called construct.
\end{definition}

\begin{example}
As an example, in the category $\mathcal{R}$ of all rings we have
$e_{\mathcal{R}}=\{e_{\mathcal{R}}^{(1)}={+}_{\mathcal{R}},
e_{\mathcal{R}}^{(2)}={\cdot}_{\mathcal{R}}\}$. Hence, if $A$ and $B$ are rings,
 $e_{A}=\{e_{A}^{(1)}={+}_{A}, e_{A}^{(2)}={\cdot}_{A}\}$ and
$e_{B}=\{e_{B}^{(1)}={+}_{B}, e_{B}^{(2)}={\cdot}_{B}\}$, where ${+}_{A}$
and ${+}_{B}$ have the same properties as ${+}_{\mathcal{R}}$
and ${\cdot}_{A}$ and ${\cdot}_{B}$ have the same properties as
${\cdot}_{\mathcal{R}}$.
\end{example}

Here, we define \emph{generalized underlying functor}.

\begin{definition}\label{defnew9}
Let $\mathcal{A}$ be a construct. A generalized underlying (GU) functor
$U_{\mathcal{A}}$ with domain $\mathcal{A}$ is a functor
$U_{\mathcal{A}}: \mathcal{A}\longrightarrow \mathcal{S}$ such that, for each
morphism $A\xrightarrow{f} A^{*}$ one has
$U_{\mathcal{A}}(A\xrightarrow{f} A^{*})= \underline{A}\xrightarrow
{U_{\mathcal{A}}(f)} {\underline{A}}^{*}$,
where $\underline{A}$ is the $\mathcal{A}$-object $A$ considered only
as its underlying set (without any operations) and $U_{\mathcal{A}}(f)$ is any
function from $\underline{A}$ to ${\underline{A}}^{*}$.
\end{definition}

\begin{remark}
Note that the usual underlying (or forgetful) functor
(see Example $3.20 (2)$, pg. $30$ in \cite{strecker:1990})
is a particular case of the GU functor shown in Definition~\ref{defnew9}.
In fact, besides forgetting the internal structure, a GU functor can also
forget the injectivity of morphisms.
\end{remark}

In the following, we introduce the definition of emergence, the central
concept of this work.

\begin{definition}\label{defnew10}
An emergence is an ordered triple $ {\mathcal{E}}_{\mathcal{A}}=
( \mathcal{A}, e_{\mathcal{A}}, U_{\mathcal{A}})$,
where $\mathcal{A}$ is a construct, $e_{\mathcal{A}}$ is a finite set of
operations and $U_{\mathcal{A}}$ is a GU functor
$U_{\mathcal{A}}:\mathcal{A}\longrightarrow \mathcal{S}$. The order
$\operatorname{o}({\mathcal{E}}_{\mathcal{A}})$ of
$ {\mathcal{E}}_{\mathcal{A}}$ is the cardinality of $e_{\mathcal{A}}$.
\end{definition}

Note that such definition corroborates with the concept of emergent
phenomena presented in biological systems in the sense that,
for the same mass, the organism present more properties that its parts
separately. More specifically, a whole (\emph{i. e.}, a system)
is regarded as a construct, while its material composition without
internal structures is regarded as the category $\mathcal{S}$.
We will explain more carefully this fact throughout
the paper (see Sections~\ref{sec5}~and~\ref{sec6}).

\begin{proposition}\label{newteo16}
Let $ {\mathcal{E}}_{\mathcal{A}}= ( \mathcal{A}, e_{\mathcal{A}},
U_{\mathcal{A}} )$ be an emergence. Then
${\mathcal{E}}_{{\mathcal{A}}^{op}}=( {\mathcal{A}}^{op}, e_{{\mathcal{A}}^{op}},$
$U_{{\mathcal{A}}^{op}} )$
is also an emergence. Moreover, $\operatorname{o}({\mathcal{E}}_{\mathcal{A}})=
\operatorname{o}({\mathcal{E}}_{{\mathcal{A}}^{op}})$.
\end{proposition}

\begin{proof}
Since $\mathcal{A}$ is a construct, it follows that ${\mathcal{A}}^{op}$
is also a construct. Considering the usual underlying functor
it follows that $U_{{\mathcal{A}}^{op}}: {\mathcal{A}}^{op}\longrightarrow
{\mathcal{S}}$ is an emergence.
From definition of ${\mathcal{A}}^{op}$, it implies that both
orders are equal.
\end{proof}

Recall that a category $\mathcal{A}$ is said to be \emph{small} if its
class of objects $\operatorname{Ob}(\mathcal{A})$ is a set. This fact
gives rise to the following definition.

\begin{definition}\label{defsmallemerge}
Let $ {\mathcal{E}}_{\mathcal{A}}= ( \mathcal{A}, e_{\mathcal{A}},
U_{\mathcal{A}} )$ be an emergence. We say that
${\mathcal{E}}_{\mathcal{A}}$ is small if the construct $\mathcal{A}$ is small.
\end{definition}

Recall that every \emph{pre-ordered class} $(\mathfrak{X}, \leq )$ (\emph{i.e.},
$\mathfrak{X}$ is a class and $\leq$ is a reflexive and transitive
relation on $\mathfrak{X}$) gives rise to a category $C(\mathfrak{X}, \leq)$
whose objects are members of $\mathfrak{X}$, and the morphisms,
identities and compositions are given respectively by:
$\operatorname{hom}(x, y)= \{ (x, y)\}$, if $x \leq y$, and
$\operatorname{hom}(x, y)= \emptyset$, otherwise; ${id}_{x}=(x, x)$;
$(y, z) \circ (x, y) = (x, z)$.
A category $\mathcal{A}$ is said to be \emph{thin} if it
is isomorphic to a category of the form $C(\mathfrak{X}, \leq )$. In
this context we can define thin emergences.

\begin{definition}\label{defthinemer}
Let $ {\mathcal{E}}_{\mathcal{A}}= ( \mathcal{A}, e_{\mathcal{A}},
U_{\mathcal{A}} )$ be an emergence. We say that
$ {\mathcal{E}}_{\mathcal{A}}$ is thin if the construct $\mathcal{A}$ is thin.
\end{definition}

\subsection{Homomorphism and strong homomorphism}\label{substronghoeme}

In this subsection we introduce the concept of (strong) homomorphism
among emergences. Homomorphisms provide a mathematical tool in order to
predict in which cases two emergences are correlated. We start first with
the definition of homomorphism.

\begin{definition}\label{defnew14A}
Let $ {\mathcal{E}}_{\mathcal{A}}= ( \mathcal{A}, e_{\mathcal{A}},
U_{\mathcal{A}} )$ and $ {\mathcal{E}}_{\mathcal{B}}=
( \mathcal{B}, e_{\mathcal{B}}, U_{\mathcal{B}} )$ be two emergences. We
say that ${\mathcal{E}}_{\mathcal{A}}$ is homomorphic to
${\mathcal{E}}_{\mathcal{B}}$ if there exists a functor $F:\mathcal{A}
\longrightarrow\mathcal{B}$ (called homomorphism) such that
$U_{\mathcal{B}} \circ F = U_{\mathcal{A}}$. We write
${\mathcal{E}}_{\mathcal{A}} {\sim}_{h} {\mathcal{E}}_{\mathcal{B}}$
to denote that $ {\mathcal{E}}_{\mathcal{A}}$ is homomorphic
to ${\mathcal{E}}_{\mathcal{B}}$. In other words, the following
diagram
\begin{center}
\begin{tikzpicture}
  \matrix (m) [matrix of math nodes,row sep=3em,column sep=3em,minimum width=2em]
  {
     \mathcal{A} & \mathcal{S} \\
     \mathcal{B} &  \\};
  \path[-stealth]
    (m-1-1) edge node [left] {$F$} (m-2-1)
            edge node [above] {$U_{\mathcal{A}}$} (m-1-2)
    (m-2-1) edge node [below] {$U_{\mathcal{B}}$}(m-1-2);
\end{tikzpicture}
\end{center}
commutes.
Sometimes we write $F: {\mathcal{E}}_{\mathcal{A}}\longrightarrow
{\mathcal{E}}_{\mathcal{B}}$ meaning the existence of a homomorphism
$F:\mathcal{A} \longrightarrow\mathcal{B}$.
\end{definition}

It is important to observe that our definition of homomorphism is
similar to that of concrete functor (see Definition 5.9 in
\cite{strecker:1990}). However, there exist a crucial differences between
them. Since the generalized underlying functor (see
Definition~\ref{defnew9}) is not necessary faithful,
it follows that a homomorphism between emergences (according to
Definition~\ref{defnew14A}) is not always faithful, in
contrast with a concrete functor (see Proposition 5.10 (1) in
\cite{strecker:1990}). In fact, our definition generalizes the
former.

\begin{proposition}\label{homoembed}
Let ${\mathcal{E}}_{\mathcal{A}}$ and ${\mathcal{E}}_{\mathcal{B}}$
be emergences. Assume that
$F, G: {\mathcal{E}}_{\mathcal{A}}\longrightarrow {\mathcal{E}}_{\mathcal{B}}$
are homomorphisms. If $U_{\mathcal{B}}$ is embedding then $F=G$.
\end{proposition}
\begin{proof}
Since $F$ and $G$ are homomorphisms, it follows that $U_{\mathcal{B}}
\circ F=U_{\mathcal{A}}$ and $U_{\mathcal{B}}\circ G =
U_{\mathcal{A}}$; so $U_{\mathcal{B}}\circ F = U_{\mathcal{B}}\circ G$.
Because $U_{\mathcal{B}}$ is both faithful and injective on objects,
the result follows. The proof is complete.
\end{proof}

\begin{proposition}\label{prophomom}
Let $ {\mathcal{E}}_{\mathcal{A}}= ( \mathcal{A}, e_{\mathcal{A}}, U_{\mathcal{A}} )$,
$ {\mathcal{E}}_{\mathcal{B}}= ( \mathcal{B}, e_{\mathcal{B}}, U_{\mathcal{B}} )$ and
$ {\mathcal{E}}_{\mathcal{C}}= ( \mathcal{C}, e_{\mathcal{C}}, U_{\mathcal{C}} )$ be emergences.
Then the following holds:
\begin{itemize}
\item [ $\operatorname{(1)}$] the relation ${\sim}_{h}$ is reflexive, that is,
$ {\mathcal{E}}_{\mathcal{A}}{\sim}_{h}{\mathcal{E}}_{\mathcal{A}}$ for all
emergences ${\mathcal{E}}_{\mathcal{A}}$;

\item [ $\operatorname{(2)}$] the relation ${\sim}_{h}$ is transitive:
if $ {\mathcal{E}}_{\mathcal{A}} {\sim}_{h} {\mathcal{E}}_{\mathcal{B}}$ and
${\mathcal{E}}_{\mathcal{B}} {\sim}_{h} {\mathcal{E}}_{\mathcal{C}}$ then
$ {\mathcal{E}}_{\mathcal{A}} {\sim}_{h} {\mathcal{E}}_{\mathcal{C}}$.
\end{itemize}
\end{proposition}
\begin{proof}
$\operatorname{(1)}$ In order to show Item~1, note that, given an
emergence ${\mathcal{E}}_{\mathcal{A}}$, the identity functor
${Id}_{\mathcal{A}}: \mathcal{A}\longrightarrow \mathcal{A}$ satisfies
$U_{\mathcal{A}}\circ {Id}_{\mathcal{A}} =U_{\mathcal{A}}$ so
${\mathcal{E}}_{\mathcal{A}}{\sim}_{h}{\mathcal{E}}_{\mathcal{A}}$.\\
$\operatorname{(2)}$ Assume that $ {\mathcal{E}}_{\mathcal{A}}
{\sim}_{h} {\mathcal{E}}_{\mathcal{B}}$ and
${\mathcal{E}}_{\mathcal{B}}{\sim}_{h} {\mathcal{E}}_{\mathcal{C}}$.
Thus, there exist functors $T_1 :\mathcal{A}\longrightarrow
\mathcal{B}$ and $T_2 :\mathcal{B}\longrightarrow\mathcal{C}$
such that $U_{\mathcal{B}}\circ T_1 = U_{\mathcal{A}}$ and
$U_{\mathcal{C}}\circ T_2 = U_{\mathcal{B}}$. We know that
$T_1 \circ T_2$ is also a functor. Moreover, we have
$[U_{\mathcal{C}}\circ T_2] \circ T_1 = U_{\mathcal{B}}\circ T_1 =
{U}_{\mathcal{A}}$. Therefore, $ {\mathcal{E}}_{\mathcal{A}}
{\sim}_{h}{\mathcal{E}}_{\mathcal{C}}$ by means of the functor
$T_2 \circ T_1$.
\end{proof}

\begin{remark}
Note that although the relation ${\sim}_{h}$ is reflexive and transitive
it is not necessarily symmetric.
\end{remark}

We next introduce the concept of \emph{equivalence} of emergences. Recall
that a functor $F:\mathcal{A} \longrightarrow\mathcal{B}$ is called
\emph{isomorphism-dense} if for every $\mathcal{B}$-object $B$ there exists an
$\mathcal{A}$-object $A$ such that $F(A)$ is isomorphic to $B$.

\begin{definition}\label{defequivemer}
Let $ {\mathcal{E}}_{\mathcal{A}}= ( \mathcal{A}, e_{\mathcal{A}},
U_{\mathcal{A}} )$ and $ {\mathcal{E}}_{\mathcal{B}}=
( \mathcal{B}, e_{\mathcal{B}}, U_{\mathcal{B}} )$ be two emergences. A
homomorphism $F:\mathcal{A}\longrightarrow \mathcal{B}$ is said to
be an equivalence from ${\mathcal{E}}_{\mathcal{A}}$ to
${\mathcal{E}}_{\mathcal{B}}$ if $F$ is full, faithful, and
isomorphism-dense. If there exists an equivalence from
$\mathcal{A}$ to $\mathcal{B}$ we say that
${\mathcal{E}}_{\mathcal{A}}$ is equivalent to
${\mathcal{E}}_{\mathcal{B}}$.
\end{definition}

The next result shows that equivalences of emergences are
reflexive and transitive.

\begin{proposition}\label{propequivaemrge}
The equivalence among emergences is a relation reflexive and
transitive in the cartesian product of
the conglomerate of all categories and the conglomerate
of all functors.
\end{proposition}
\begin{proof}
Let $ {\mathcal{E}}_{\mathcal{A}}= ( \mathcal{A}, e_{\mathcal{A}},
U_{\mathcal{A}} )$ be an emergence. To prove the reflexivity,
it suffices to consider the identity functor ${Id}_{\mathcal{A}}:\mathcal{A}
\longrightarrow\mathcal{A}$.

To prove the transitivity, assume that $ {\mathcal{E}}_{\mathcal{A}}$ is
equivalent to $ {\mathcal{E}}_{\mathcal{B}}$ and
$ {\mathcal{E}}_{\mathcal{B}}$ is equivalent to
${\mathcal{E}}_{\mathcal{C}}$ with functors $T_1$ and $T_2$, respectively.
From Proposition~\ref{prophomom}, it follows
that $T_{2}\circ T_1$ is a homomorphism from
${\mathcal{E}}_{\mathcal{A}}$ to
${\mathcal{E}}_{\mathcal{C}}$. Proposition~3.36 in
\cite{strecker:1990} implies that the composite
$T_{2}\circ T_1 :\mathcal{A}\longrightarrow\mathcal{C}$ is also an
equivalence. Hence, the results follows.
\end{proof}

Note that the equivalence among emergences is not necessarily symmetric.
In the following, we introduce the definition of \emph{isomorphism}
for emergences.

\begin{definition}\label{defnew14}
Let $ {\mathcal{E}}_{\mathcal{A}}= ( \mathcal{A}, e_{\mathcal{A}},
U_{\mathcal{A}} )$ and $ {\mathcal{E}}_{\mathcal{B}}= ( \mathcal{B},
e_{\mathcal{B}}, U_{\mathcal{B}} )$ be two emergences.
We say that ${\mathcal{E}}_{\mathcal{A}}$ and ${\mathcal{E}}_{\mathcal{B}}$
are isomorphic, written ${\mathcal{E}}_{\mathcal{A}} {\cong}_{e}
{\mathcal{E}}_{\mathcal{B}}$, if:
\begin{itemize}
\item [ $\operatorname{(1)}$] $\mathcal{A}$ is isomorphic to $\mathcal{B}$
as categories;

\item [ $\operatorname{(2)}$] there exists an isomorphism $F:\mathcal{A}
\longrightarrow \mathcal{B}$ such that $U_{\mathcal{B}} \circ F =
U_{\mathcal{A}}$.
\end{itemize}
In other words, the following diagram
\begin{center}
\begin{tikzpicture}
  \matrix (m) [matrix of math nodes,row sep=3em,column sep=3em,minimum width=2em]
  {
     \mathcal{A} & \mathcal{S} \\
     \mathcal{B} &  \\};
  \path[-stealth]
    (m-1-1) edge node [left] {$F$} (m-2-1)
    (m-1-1) edge node [right] {${\cong}_{e}$} (m-2-1)
            edge node [above] {$U_{\mathcal{A}}$} (m-1-2)
    (m-2-1) edge node [below] {$U_{\mathcal{B}}$}(m-1-2);
\end{tikzpicture}
\end{center} commutes.
\end{definition}

\begin{proposition}\label{isoehequiv}
Let ${\mathcal{E}}_{\mathcal{A}}$ and ${\mathcal{E}}_{\mathcal{B}}$ be
two emergences. If ${\mathcal{E}}_{\mathcal{A}} {\cong}_{e}
{\mathcal{E}}_{\mathcal{B}}$ then both ${\mathcal{E}}_{\mathcal{A}}$ and
${\mathcal{E}}_{\mathcal{B}}$ are equivalent to each other.
\end{proposition}
\begin{proof}
The proof is immediate, so it is omitted.
\end{proof}

In the case in which the respective generalized underlying
functors of two homomorphic emergences are
embeddings, one has an isomorphism between them.

\begin{proposition}\label{prophomoiso}
Let $ {\mathcal{E}}_{\mathcal{A}}= ( \mathcal{A}, e_{\mathcal{A}},
U_{\mathcal{A}} )$ and $ {\mathcal{E}}_{\mathcal{B}}=
(\mathcal{B}, e_{\mathcal{B}}, U_{\mathcal{B}} )$ be two emergences such that
${\mathcal{E}}_{\mathcal{A}} {\sim}_{h} {\mathcal{E}}_{\mathcal{B}}$ and
${\mathcal{E}}_{\mathcal{B}} {\sim}_{h} {\mathcal{E}}_{\mathcal{A}}$.
If $U_{\mathcal{A}}$ and $U_{\mathcal{B}}$ are both embeddings, then
$ {\mathcal{E}}_{\mathcal{A}} {\cong}_{e} {\mathcal{E}}_{\mathcal{B}}$.
\end{proposition}
\begin{proof}
From hypotheses, there exists a homomorphism
$F_1 : \mathcal{A}\longrightarrow\mathcal{B}$ such that $U_{\mathcal{B}}
\circ F_1 = U_{\mathcal{A}}$. We must show that $F_1$ is an isomorphism from
$\mathcal{A}$ to $\mathcal{B}$.
Again, from hypotheses, there exists a functor $F_2 : \mathcal{B}
\longrightarrow\mathcal{A}$ such that
$U_{\mathcal{A}}\circ F_2 = U_{\mathcal{B}}$. Thus $U_{\mathcal{A}}
\circ (F_2 \circ F_1 ) = U_{\mathcal{B}}
\circ F_1 = U_{\mathcal{A}}=U_{\mathcal{A}}
\circ {Id}_{\mathcal{A}}$. Since $U_{\mathcal{A}}$ is embedding,
it follows that $U_{\mathcal{A}}$ is also faithful and
injective on objects; thus $F_2 \circ F_1 = {Id}_{\mathcal{A}}$. On the other
hand we have $U_{\mathcal{B}}\circ (F_1 \circ F_2 )= U_{\mathcal{A}}
\circ F_2 = U_{\mathcal{B}}$. As $U_{\mathcal{B}}$ is also an embedding,
one has $F_1 \circ F_2 = {Id}_{\mathcal{B}}$.
Thus, $F_1$ is an isomorphism from $\mathcal{A}$ to $\mathcal{B}$.
Therefore, $ {\mathcal{E}}_{\mathcal{A}} {\cong}_{e} {\mathcal{E}}_{\mathcal{B}}$,
as required. The proof is complete.
\begin{eqnarray*}
\begin{tikzpicture}
  \matrix (m) [matrix of math nodes,row sep=3em,column sep=3em,minimum width=2em]
  {
     \mathcal{A} & \mathcal{S} \\
     \mathcal{B} &  \\};
  \path[-stealth]
    (m-1-1) edge node [left] {$F_1$} (m-2-1)
            edge node [above] {$U_{\mathcal{A}}$} (m-1-2)
    (m-2-1) edge node [below] {$U_{\mathcal{B}}$}(m-1-2);
\end{tikzpicture}
\quad
\begin{tikzpicture}
  \matrix (m) [matrix of math nodes,row sep=3em,column sep=3em,minimum width=2em]
  {
     \mathcal{A} & \mathcal{S} \\
     \mathcal{B} &  \\};
  \path[-stealth]
   (m-1-1)  edge node [above] {$U_{\mathcal{A}}$} (m-1-2)
   (m-2-1) edge node [left] {$F_2$} (m-1-1)
   (m-2-1) edge node [below] {$U_{\mathcal{B}}$}(m-1-2);
\end{tikzpicture}
\end{eqnarray*}
\end{proof}

\begin{proposition}\label{newteo15}
The isomorphism among emergences ${\cong}_{e}$ is an equivalence
relation in the cartesian product of the conglomerate of
all categories and the conglomerates of all functors between constructs.
\end{proposition}
\begin{proof}
Let $ {\mathcal{E}}_{\mathcal{A}}= ( \mathcal{A}, e_{\mathcal{A}},
U_{\mathcal{A}} )$ be an emergence. We first show that
${\mathcal{E}}_{\mathcal{A}}{\cong}_{e} {\mathcal{E}}_{\mathcal{A}}$. It is clear
that $\mathcal{A}\cong \mathcal{A}$ as categories. Moreover,
${id}_{\mathcal{A}}: \mathcal{A}\longrightarrow \mathcal{A}$ is
an isomorphism and $U_{\mathcal{A}} \circ {id}_{\mathcal{A}}=
U_{\mathcal{A}}$. So ${\cong}_{e}$ is reflexive.

Assume that
${\mathcal{E}}_{\mathcal{A}}= ( \mathcal{A}, e_{\mathcal{A}},
U_{\mathcal{A}} )$ and $ {\mathcal{E}}_{\mathcal{B}}=
( \mathcal{B}, e_{\mathcal{B}}, U_{\mathcal{B}} )$
are two emergences such that
${\mathcal{E}}_{\mathcal{A}}{\cong}_{e}{\mathcal{E}}_{\mathcal{B}}$.
Thus $\mathcal{A}\cong \mathcal{B}$ as categories implies
that $\mathcal{B}\cong \mathcal{A}$ as categories.
Since there exists an isomorphism $F: \mathcal{A}
\longrightarrow\mathcal{B}$ from $\mathcal{A}$ to $\mathcal{B}$
satisfying $U_{\mathcal{B}}\circ F =
U_{\mathcal{A}}$, it follows that $U_{\mathcal{B}}=U_{\mathcal{A}}
\circ F^{-1}$. Hence ${\cong}_{e}$ is symmetric.

To prove the transitivity, let us consider that ${\mathcal{E}}_{\mathcal{A}}$,
${\mathcal{E}}_{\mathcal{B}}$ and ${\mathcal{E}}_{\mathcal{C}}$
are emergences, with ${\mathcal{E}}_{\mathcal{A}} {\cong}_{e}
{\mathcal{E}}_{\mathcal{B}}$ and ${\mathcal{E}}_{\mathcal{B}} {\cong}_{e}
{\mathcal{E}}_{\mathcal{C}}$. It is clear that $\mathcal{A} \cong
\mathcal{C}$ as categories. Additionally, we know that there exist
homomorphisms $F :\mathcal{A}\longrightarrow \mathcal{B}$ and $G:\mathcal{B}
\longrightarrow\mathcal{C}$ such that $U_{\mathcal{B}} \circ F =
U_{\mathcal{A}}$ and  $U_{\mathcal{C}}\circ G = U_{\mathcal{B}}$. Thus,
we have $(U_{\mathcal{C}} \circ G) \circ F = U_{\mathcal{A}}\Longrightarrow
U_{\mathcal{C}} \circ (G \circ F) = U_{\mathcal{A}}$. Since $G\circ
F:\mathcal{A}\longrightarrow \mathcal{C}$ is also an isomorphism, it
follows that the relation ${\cong}_{e}$ is transitive. This completes
the proof.
\end{proof}

We next define the concept of \emph{strong homomorphism}. As the name
proposed, it is more powerful to correlate emergences than
homomorphisms.

\begin{definition}\label{defnew14AS}
Let $ {\mathcal{E}}_{\mathcal{A}}= ( \mathcal{A}, e_{\mathcal{A}},
U_{\mathcal{A}} )$ and $ {\mathcal{E}}_{\mathcal{B}}=
(\mathcal{B}, e_{\mathcal{B}}, U_{\mathcal{B}} )$ be two emergences. We
say that ${\mathcal{E}}_{\mathcal{A}}$ is strong homomorphic to
${\mathcal{E}}_{\mathcal{B}}$ if there exists a functor
$F:\mathcal{A} \longrightarrow\mathcal{B}$ (called strong
homomorphism) such that $U_{\mathcal{B}} \circ F = U_{\mathcal{A}}$
and $|e_{\mathcal{A}}|= |e_{\mathcal{B}}|$. We write
${\mathcal{E}}_{\mathcal{A}}{\sim}_{h}^{st} {\mathcal{E}}_{\mathcal{B}}$
to denote that $ {\mathcal{E}}_{\mathcal{A}}$ is strong
homomorphic to ${\mathcal{E}}_{\mathcal{B}}$.
\end{definition}

The following result is a natural consequence of the previous definition.

\begin{proposition}\label{prophomomS}
The relation ${\sim}_{h}^{st}$ is reflexive and transitive.
\end{proposition}
\begin{proof}
Similar to that of Proposition~\ref{prophomom}.
\end{proof}

Note that, as in the case of homomorphisms, strong homomorphisms are not
necessarily symmetric. The next definition establishes when two
emergences are strong isomorphic.

\begin{definition}\label{defnew14S}
Let $ {\mathcal{E}}_{\mathcal{A}}= ( \mathcal{A}, e_{\mathcal{A}},
U_{\mathcal{A}} )$ and $ {\mathcal{E}}_{\mathcal{B}}= ( \mathcal{B},
e_{\mathcal{B}}, U_{\mathcal{B}} )$ be two emergences.
We say that ${\mathcal{E}}_{\mathcal{A}}$ and ${\mathcal{E}}_{\mathcal{B}}$
are strong isomorphic written ${\mathcal{E}}_{\mathcal{A}} {\cong}_{e}^{st}
{\mathcal{E}}_{\mathcal{B}}$ if
\begin{itemize}
\item [ $\operatorname{(1)}$] $\mathcal{A}$ is isomorphic to
$\mathcal{B}$ as categories;

\item [ $\operatorname{(2)}$] there exists an isomorphism $F:\mathcal{A}
\longrightarrow\mathcal{B}$ such
that $U_{\mathcal{B}} \circ F = U_{\mathcal{A}}$;

\item [ $\operatorname{(3)}$] $|e_{\mathcal{A}}|=|e_{\mathcal{B}}|$.
\end{itemize}
\end{definition}

\begin{proposition}\label{propisomoS}
The relation ${\cong}_{e}^{st}$ is an equivalence relation.
\end{proposition}
\begin{proof}
Similar to that of Proposition~\ref{newteo15}.
\end{proof}

\subsection{Cartesian product of emergence}\label{subcarteeme}

In this section, we show that cartesian products of emergences are also emergences.
Let us first recall the \emph{cartesian product} of categories.

\begin{definition}\label{cartesian}
Let $\mathcal{A}$ and $\mathcal{B}$ be two categories. The cartesian product
of $\mathcal{A}$ and $\mathcal{B}$ is the category $\mathcal{A}
\times \mathcal{B}$ defined as follows: the objets are all ordered pairs $(A, B)$,
where $ A \in \operatorname{Ob}(\mathcal{A})$ and $B \in \operatorname{Ob}
(\mathcal{B})$; for all $(A, B), (A^{*}, B^{*}) \in \operatorname{Ob}(\mathcal{A})
\times \operatorname{Ob}(\mathcal{B})$, given $\mathcal{A}$-morphisms
$A \xrightarrow{f_{1}} A^{*}$ and $A^{*} \xrightarrow{f_{2}} A^{'}$ and
$\mathcal{B}$-morphisms $B \xrightarrow{g_{1}} B^{*}$ and $B^{*}
\xrightarrow{g_{2}} B^{'}$ ,
the composition is defined as $(f_2 , g_2 )\circ(f_1 , g_1 )=(f_2
\circ f_1 , g_2 \circ g_1 )$, and the
identities are $id_{(A, B)}=(id_{A}, id_{B})$.
This concept can be generalized similarly to cartesian product of $n$
categories.
\end{definition}

Although the following result is well-known, we prove it here for completeness.

\begin{proposition}\label{newprop20}
Let ${\mathcal{A}}_{1}, {\mathcal{A}}_{2}, \ldots , {\mathcal{A}}_{n}$ and
${\mathcal{B}}_{1}, {\mathcal{B}}_{2}, \ldots , {\mathcal{B}}_{n}$ be categories and
consider the cartesian products ${\mathcal{A}}_{1} \times {\mathcal{A}}_{2}\times \ldots
\times {\mathcal{A}}_{n}$ and
${\mathcal{B}}_{1} \times {\mathcal{B}}_{2}\times \ldots \times {\mathcal{B}}_{n}$.
Assume that $F={\prod}_{i=1}^{n}F_{i}:
{\mathcal{A}}_{1} \times {\mathcal{A}
}_{2}\times \ldots \times {\mathcal{A}}_{n}\longrightarrow
{\mathcal{B}}_{1} \times {\mathcal{B}}_{2}\times \ldots \times {\mathcal{B}}_{n}$
is the application of the corresponding $F_i$ componentwise,
where $F_i: {\mathcal{A}}_{i}\longrightarrow
{\mathcal{B}}_{i}$ is a functor from ${\mathcal{A}}_{i}$ to
${\mathcal{B}}_{i}$ for each $i=1, \ldots , n$. Then $F$ is functor.
\end{proposition}

\begin{proof}
Assume that ${(A_{i})}_{i=1}^{n} \in \operatorname{Ob}({\prod}_{i=1}^{n}
{{\mathcal{A}}_{i}})$. From definition we have
$ F({(A_{i})}_{i=1}^{n})= (F_{1}(A_{1}) , F_{2}(A_{2}), \ldots , F_{n}(A_{n})) \in
\operatorname{Ob}({\prod}_{i=1}^{n} {{\mathcal{B}}_{i}}).$
Suppose that $( f_1 , f_2 , \ldots , f_n ), ( g_1 ,$ $g_2 , \ldots , g_n ) \in
\operatorname{Mor} ({\prod}_{i=1}^{n} {{A}}_{i})$, where $f_{i}:A_{i}^{(1)}
\longrightarrow A_{i}^{(2)}$ and
$g_{i}:A_{i}^{(2)}\longrightarrow A_{i}^{(3)}$ for all $i=1, 2, \ldots , n$.
Thus
\begin{eqnarray*}
F[( g_1 , g_2 , \ldots , g_n ) \circ ( f_1 , f_2 , \ldots , f_n )]&=&
F( [g_{1}\circ f_{1}] , [g_{2}\circ f_{2}], \ldots , [g_{n}\circ f_{n}] )\\
&\stackrel{def}{=}& ( F_{1}(g_{1}\circ f_{1}) , F_{2}(g_{2}\circ f_{2}),
\ldots , F_{n}(g_{n}\circ f_{n}))\\
&=& ([F_{1}(g_{1}) \circ F_{1}(f_{1})] , [F_{2}(g_{2}) \circ F_{2}(f_{2})],
\ldots ,\\ & & [F_{n}(g_{n}) \circ F_{n}(f_{n})]\\
&\stackrel{def}{=}& (F_{1}(g_{1}) , F_{2}(g_{2}) , \ldots , F_{n}(g_{n}) )
\circ \\ & & (F_{1}(f_{1}) , F_{2}(f_{2}) ,
\ldots , F_{n}(f_{n}) )\\
&=& F( g_1 , g_2 , \ldots , g_n )\circ F( f_1 , f_2 , \ldots , f_n ).
\end{eqnarray*}

Let ${(A_{i})}_{i=1}^{n} \in \operatorname{Ob}({\prod}_{i=1}^{n} {{\mathcal{A}}_{i}})$
and let ${id}_{A_{i}}: A_{i} \longrightarrow A_{i}$
be the identity of $A_{i}$ for all $i= 1, 2, \ldots , n$. Therefore,
\begin{eqnarray*}
F( {id}_{A_{1}}, {id}_{A_{2}}, \ldots , {id}_{A_{n}}) &\stackrel{def}{=}&
( F_1 ({id}_{A_{1}}), F_2 ({id}_{A_{2}}) , \ldots , F_n ({id}_{A_{n}}) )\\
&=& ( {id}_{F_1 (A_{1})}, {id}_{F_2 (A_{2})}, \ldots , {id}_{F_n (A_{n})})\\
&\stackrel{def}{=}& {Id}_{F({\prod}_{i=1}^{n} {{A}}_{i})}.
\end{eqnarray*}
\end{proof}

Let ${\mathcal{A}}_{1}, {\mathcal{A}}_{2},
\ldots , {\mathcal{A}}_{n}$ be categories. The projector functor
${\Pi}_{{\mathcal{A}}_{j}}: {\prod}_{i=1}^{n} {{\mathcal{A}}_{i}}\longrightarrow
{\mathcal{A}}_{j}$ is defined as follows: if $ {(A_{i})}_{i=1}^{n}\in
\operatorname{Ob}({\prod}_{i=1}^{n} {{\mathcal{A}}_{i}})$ and if
${(f_{i})}_{i=1}^{n}\in \operatorname{Mor}({\prod}_{i=1}^{n} {{\mathcal{A}}_{i}})$
then $P_{{\mathcal{A}}_{j}} ({(A_{i})}_{i=1}^{n})= A_j$ and
$P_{{\mathcal{A}}_{j}} ({(f_{i})}_{i=1}^{n})=f_{j}$.

The following result is a natural consequence of our definition of emergence.

\begin{proposition}\label{newprop22}
Let ${\mathcal{A}}_{1}, {\mathcal{A}}_{2}, \ldots , {\mathcal{A}}_{n}$ be constructs.
Then the triple $({\prod}_{i=1}^{n} {{\mathcal{A}}_{i}},$
$e_{({\prod}_{i=1}^{n} {{\mathcal{A}}_{i}})}, U_{({\prod}_{i=1}^{n}
{{\mathcal{A}}_{i}})})$, where $U_{({\prod}_{i=1}^{n} {{\mathcal{A}}_{i}})}$ is the
usual underlying functor, is an emergence.
\end{proposition}
\begin{proof}
Since the cartesian product of
constructs is also a construct (the operations are the operations
of each construct which composes it) and because
$U_{({\prod}_{i=1}^{n} {{\mathcal{A}}_{i}})}$ is a GU functor,  we are done.
\end{proof}

\subsection{Representability and sub-emergence}\label{subrepresenteme}

Representability of emergences is defined according to the representability of
their corresponding GU functors.

\begin{definition}\label{emerepres}
Let ${\mathcal{E}}_{\mathcal{A}}= ( \mathcal{A}, e_{\mathcal{A}}, U_{\mathcal{A}} )$
be an emergence. We say that ${\mathcal{E}}_{\mathcal{A}}$ is representable
if the GU functor $U_{\mathcal{A}}:\mathcal{A}\longrightarrow \mathcal{S}$ is representable.
\end{definition}

\begin{example}
Let $\mathcal{V}$ be the construct of the real vector
spaces where the GU functor is the usual underlying functor.
Then the emergence ${\mathcal{E}}_{\mathcal{V}}$
is represented by the pair $(\mathbb{R}, 1)$. If we have
$\mathcal{A}=\mathcal{G}$, that is,
$\mathcal{A}$ is the construct of groups, where the GU
functor is considered as the usual underlying functor,
then the emergence
${\mathcal{E}}_{\mathcal{G}}$ is representable by
the pair $(\mathbb{Z}, 1)$.
\end{example}

\begin{proposition}\label{repgumono}
If the emergence ${\mathcal{E}}_{\mathcal{A}}$ is representable,
then $U_{\mathcal{A}}$ preserves monomorphisms.
\end{proposition}
\begin{proof}
See Proposition~\cite[Proposition 7.37 (1)]{strecker:1990}.
\end{proof}

The concept of \emph{sub-emergence} is defined similarly as subcategory.

\begin{definition}\label{sub-emergence}
Let $ {\mathcal{E}}_{\mathcal{A}}= ( \mathcal{A}, e_{\mathcal{A}},
U_{\mathcal{A}} )$ be an emergence.
We say that ${\mathcal{E}}_{\mathcal{B}}= ( \mathcal{B}, e_{\mathcal{B}},
U_{\mathcal{B}} )$ is a sub-emergence of $ {\mathcal{E}}_{\mathcal{A}}$
if the following conditions hold:
\begin{itemize}
\item [ $\operatorname{(i)}$] $\operatorname{Ob}(\mathcal{B})\subseteq
\operatorname{Ob}(\mathcal{A})$;

\item[ $\operatorname{(ii)}$] for all $A, A^{*} \in \operatorname{Ob}
(\mathcal{B})$, it follows that
${\operatorname{hom}}_{\mathcal{B}}(A, A^{*})\subseteq {\operatorname{hom}}_{\mathcal{A}}
(A, A^{*}) $;

 \item [ $\operatorname{(iii)}$] for all $A \in \operatorname{Ob}
(\mathcal{B})$, there exists ${id}_{A}: A\longrightarrow A$
such that ${id}_{A}$ viewed as an $\mathcal{B}$-morphism is equal to ${id}_{A}$
viewed as an $\mathcal{A}$-morphism;

\item [ $\operatorname{(iv)}$] the composition law in $\mathcal{B}$
is the restriction of the composition law in $\mathcal{A}$;

\item[ $\operatorname{(v)}$] the inclusion functor $E:\mathcal{B}\hookrightarrow
\mathcal{A}$ satisfies $U_{\mathcal{B}}= U_{\mathcal{A}}
\circ E$.
\end{itemize}
\end{definition}

\begin{remark}
$\operatorname{(1)}$ Notice that, since $\operatorname{Ob}(\mathcal{B})
\subseteq \operatorname{Ob}(\mathcal{A})$, it follows that
$e_{\mathcal{A}}=e_{\mathcal{B}}$.\\
$\operatorname{(2)}$ It is clear from the definition that all
sub-emergence is also an emergence.
\end{remark}

\begin{definition}\label{fullsub-emer}
Let ${\mathcal{E}}_{\mathcal{B}}= ( \mathcal{B}, e_{\mathcal{B}},
U_{\mathcal{B}})$ be a sub-emergence of
${\mathcal{E}}_{\mathcal{A}}$. We say that ${\mathcal{E}}_{\mathcal{B}}$ is a
full sub-emergence if ${\mathcal{E}}_{\mathcal{B}}$
satisfies the criteria given in
Definition~\ref{sub-emergence} and also the following condition:
for all $A, A^{*} \in \operatorname{Ob}(\mathcal{A})$, it
follows that ${\operatorname{hom}}_{\mathcal{A}}(A, A^{*})=
{\operatorname{hom}}_{\mathcal{B}}(A, A^{*})$.
\end{definition}

For every sub-emergency ${\mathcal{E}}_{\mathcal{B}}$ of the
emergence ${\mathcal{E}}_{\mathcal{A}}$, the inclusion functor
$I_{(\mathcal{B}\hookrightarrow\mathcal{A})}: \mathcal{B}\hookrightarrow
\mathcal{A}$ is always an embedding
(injective on morphisms). Such functor is full (\emph{i.e}.,
surjective on the $\operatorname{hom}$ sets) if and only
if ${\mathcal{E}}_{\mathcal{B}}$
is a full sub-emergence of ${\mathcal{E}}_{\mathcal{A}}$.

\begin{definition}\label{newdef18}
Let $ {\mathcal{E}}_{\mathcal{A}}= ( \mathcal{A}, e_{\mathcal{A}},
U_{\mathcal{A}} )$ and $ {\mathcal{E}}_{\mathcal{B}}=
( \mathcal{B}, e_{\mathcal{B}}, U_{\mathcal{B}} )$ be two emergences. We say that
$ {\mathcal{E}}_{\mathcal{A}}$ induces $ {\mathcal{E}}_{\mathcal{B}}$ if each
$\mathcal{B}$-object $B= (\underline{B}, e_{\mathcal{B}})$ induces an $\mathcal{A}$-object
$A= (\underline{A}=\underline{B}, e_{A})$ such that $e_{B}\subseteq e_{A}$.
\end{definition}

The concept of \emph{quasi-emergence} is analogous to that of quasi-category.

\begin{definition}\label{quasiemer}
A quasi-emergence is an ordered quadruple
$\mathcal{E}\mathcal{M}= (\operatorname{Ob}(\mathcal{E}\mathcal{M}),
$ ${\operatorname{hom}}_{(\mathcal{E}\mathcal{M})}, id, \circ )$, where
$\operatorname{Ob}(\mathcal{E}\mathcal{M})$ is the conglomerate of all emergences,
${\operatorname{hom}}_{\mathcal{E}\mathcal{M}}$ is the conglomerate
of all homomorphisms between emergences, $id$ is the class of all identities
of constructs and the composite $\circ$ is the composite
in the sense of emergences according to Definition~\ref{defnew14A}.
\end{definition}

\subsection{Equalizer and co-equalizer of emergence}\label{subequalieme}

In this section we introduce the concept of equalizer for emergences.
The following definition establishes such concept.

\begin{definition}\label{defemeequa}
Let $ {\mathcal{E}}_{\mathcal{A}}$,
$ {\mathcal{E}}_{\mathcal{B}}$ and ${\mathcal{E}}_{\mathfrak{E}}$
be emergences. Assume that $F, G:\mathcal{A}\longrightarrow
\mathcal{B}$ are functors. We say that a functor $E:\mathfrak{E}\longrightarrow
\mathcal{A}$ is an equalizer emergence of $F$ and $G$ if:
\begin{itemize}
\item [ $\operatorname{(1)}$] $F \circ E = G \circ E$;

\item[ $\operatorname{(2)}$] for each emergence ${\mathcal{E}}_{{\mathfrak{E}}^{*}}$
and for each functor $E^{*}:{\mathfrak{E}}^{*}\longrightarrow \mathcal{A}$
such that $F \circ E^{*} = G \circ E^{*}$, there exists a unique functor (universal property)
$\overline{E}:{\mathfrak{E}}^{*} \longrightarrow\mathfrak{E}$ such that $E^{*}= E
\circ \overline{E}$ and the following diagram
\end{itemize}
\begin{eqnarray*}
\begin{tikzpicture}
  \matrix (m) [matrix of math nodes,row sep=3em,column sep=3em,minimum width=2em]
  {
      {\mathfrak{E}}^{*} &  & \\
      \mathfrak{E} &  \mathcal{A} & \mathcal{B}\\};
  \path[-stealth]
    (m-1-1) edge node [left] {$\overline{E}$} (m-2-1)
    (m-1-1) edge node [right] {$E^{*}$}(m-2-2)
    (m-2-1) edge node [below] {$E$} (m-2-2)
    (m-2-2) edge node [above] {$F$} (m-2-3)
    (m-2-2) edge node [below] {$G$} (m-2-3);
\end{tikzpicture}
\end{eqnarray*} commutes.
\end{definition}

In the quasi-emergence $\mathcal{E}\mathcal{M}$, equalizer emergences exist.

\begin{theorem}\label{existequaeme}
Let $ {\mathcal{E}}_{\mathcal{A}}$ and $ {\mathcal{E}}_{\mathcal{B}}$
be emergences. Assume that
$F, G:\mathcal{A}\longrightarrow \mathcal{B}$ are functors. Then there exists an
equalizer emergence $ \mathfrak{E} \xrightarrow{E} \mathcal{A}$ of $F$ and $G$.
\end{theorem}
\begin{proof}
We first define the construct $\mathfrak{E}$ as follows. The class of
objects of $\mathfrak{E}$ is the class $\operatorname{Ob}(\mathfrak{E})=
\{ A \in \operatorname{Ob}(\mathcal{A}) \ | \ F(A) = G(A) \}$; the class of morphisms of
$\mathfrak{E}$ is the class $\operatorname{Mor}(\mathfrak{E})=\{
f \in \operatorname{Mor}(\mathcal{A}) \ | \ F(f)=G(f) \}$. We will
show that $\operatorname{Mor}(\mathfrak{E})$ is well defined. In fact,
for each $\mathfrak{E}$-morphisms $A \xrightarrow{f} B$ and
$B \xrightarrow{g} C$, we have
\begin{eqnarray*}
F( g \circ f) &=& F(g)\circ F(f)\\
&=& G(g) \circ G(f)\\
&=& G(g \circ f)
\end{eqnarray*}
that is, the composite $g \circ f$ is also an
$\mathfrak{E}$-morphism. Moreover, for each $\mathfrak{E}$-object $A$
and for the identity $A \xrightarrow{{id}_{A}} A$ we have $F(A)=G(A)$ and
$[F(A) \xrightarrow{{id}_{F(A)}} F(A) ]= [G(A)
\xrightarrow{{id}_{G(A)}} G(A)]$ since $F({id}_{A})=
{id}_{F(A)}={id}_{G(A)}=G({id}_{A})$ by uniqueness of the identity. Hence, for
each $\mathfrak{E}$-object $A$, the
identity ${id}_{A}$ belongs to $\operatorname{Mor}(\mathfrak{E})$.

We claim that the inclusion functor $\mathfrak{E}\xrightarrow{{I}_{(\mathfrak{E}
\hookrightarrow\mathcal{A})}}\mathcal{A}$ is an
equalizer emergence of $F$ and $G$. In fact, it
follows from construction that $F \circ {I}_{(\mathfrak{E}
\hookrightarrow\mathcal{A})}= G \circ {I}_{(\mathfrak{E}
\hookrightarrow\mathcal{A})}$. Further, for each functor
${\mathfrak{E}}^{*} \xrightarrow{E^{*}} \mathcal{A}$
such that $F \circ E^{*} = G \circ E^{*}$ one has
${\operatorname{Im}}_{(\operatorname{Ob}({\mathfrak{E}}^{*}))}[E^{*}]
\subset \operatorname{Ob}(\mathfrak{E})$ and
${\operatorname{Im}}_{(\operatorname{Mor}({\mathfrak{E}}^{*}))}[E^{*}]\subset
\operatorname{Mor}(\mathfrak{E})$. Therefore, we define  the unique
functor $\overline{E}:{\mathfrak{E}}^{*}
\longrightarrow \mathfrak{E}$
such that $E^{*}= {I}_{(\mathfrak{E}
\hookrightarrow\mathcal{A})} \circ \overline{E}$, that is, $\overline{E}= E^{*}$.
The proof is complete.
\end{proof}

Equalizer emergences are essentially unique, as states the following result.

\begin{proposition}\label{equaemerg}
Let $ {\mathcal{E}}_{\mathcal{A}}$, $ {\mathcal{E}}_{\mathcal{B}}$
be emergences and $F, G:\mathcal{A}\longrightarrow
\mathcal{B}$ be functors. If $E:\mathfrak{E}\longrightarrow \mathcal{A}$
and $E^{*}:{\mathfrak{E}}^{*}\longrightarrow \mathcal{A}$
are equalizer emergences of $F$ and $G$, then there exists an isomorphism
$T:{\mathfrak{E}}^{*}\longrightarrow \mathfrak{E}$ such that $E^{*} = E
\circ T$. Moreover, if $E:\mathfrak{E}
\longrightarrow \mathcal{A}$ is an equalizer emergence of $F$ and $G$ and
$I:{\mathfrak{E}}^{*}\longrightarrow \mathfrak{E}$ is an
isomorphism, then the composite $E\circ I:{\mathfrak{E}}^{*}\longrightarrow
\mathcal{A}$ is also an equalizer emergence of $F$ and $G$.
\end{proposition}
\begin{proof}
It suffices to adapt the proof of Proposition~$7.53$ of \cite{strecker:1990}.
\end{proof}

We next introduce the concept of \emph{regular monomorphism emergence}.

\begin{definition}\label{reghomoeme}
Let ${\mathcal{E}}_{\mathcal{A}}$ and ${\mathcal{E}}_{\mathfrak{E}}$
be emergences. A functor $E:\mathfrak{E}\longrightarrow \mathcal{A}$
is said to be a regular monomorphism emergence if it is an equalizer emergence of
some pair of functors.
\end{definition}

The following result establishes nice properties of regular monomorphism
emergences (RHE)'s.

\begin{proposition}\label{regmonomono}
Let ${\mathcal{E}}_{\mathcal{A}}$ and ${\mathcal{E}}_{\mathfrak{E}}$
be emergences and let $E:\mathfrak{E}\longrightarrow \mathcal{A}$
be a functor. Then the following are equivalent:
\begin{itemize}

\item [ $\operatorname{(1)}$] $E:\mathfrak{E}\longrightarrow
\mathcal{A}$ is a RHE.

\item [ $\operatorname{(2)}$] $E:\mathfrak{E}\longrightarrow
\mathcal{A}$ is a monomorphism.

\item [ $\operatorname{(3)}$] $E:\mathfrak{E}\longrightarrow
\mathcal{A}$ is an embedding.
\end{itemize}
\end{proposition}
\begin{proof}

$[\operatorname{(1)}\Longrightarrow \operatorname{(2)}]$
Let $E:\mathfrak{E}\longrightarrow \mathcal{A}$
be a RHE. Then there exists a pair of functors
$F, G:\mathcal{A}\longrightarrow \mathcal{B}$ such that $E:\mathfrak{E}\longrightarrow
\mathcal{A}$ is an equalizer emergence of $F$ and $G$.
Let ${\mathcal{E}}_{\mathcal{C}}$ be an emergence and assume that $H, K:
\mathcal{C}\longrightarrow \mathfrak{E}$ are two functors such that $E\circ H = E
\circ K$. We must show that $H=K$. Let us consider $E^{*}=E\circ H = E
\circ K$. Since $E:\mathfrak{E}\longrightarrow \mathcal{A}$ is equalizer
emergence, it follows that $F \circ E = G \circ E$. Thus, we have
$F\circ E^{*}= F\circ (E\circ H)=G \circ (E \circ H) =
G \circ (E \circ K)= G \circ E^{*}$. From definition, there exists
a unique functor $\overline{E}:\mathcal{C}\longrightarrow
\mathfrak{E}$ such that $E^{*} = E \circ \overline{E}= E\circ H = E
\circ K$. Therefore $H=K$, hence $E:\mathfrak{E}\longrightarrow \mathcal{A}$
is monomorphism.
\begin{eqnarray*}
\begin{tikzpicture}
  \matrix (m) [matrix of math nodes,row sep=3em,column sep=3em,minimum width=2em]
  {
      \mathcal{C} &  & \\
      \mathfrak{E} &  \mathcal{A} & \mathcal{B}\\};
  \path[-stealth]
    (m-1-1) edge node [left] {$H$} (m-2-1)
    (m-1-1) edge node [right] {$K$} (m-2-1)
    (m-1-1) edge node [right] {$E\circ H$}(m-2-2)
    (m-2-1) edge node [below] {$E$} (m-2-2)
    (m-2-2) edge node [above] {$F$} (m-2-3)
    (m-2-2) edge node [below] {$G$} (m-2-3);
\end{tikzpicture}
\end{eqnarray*}

\vspace{0.2 cm}

$[\operatorname{(2)}\Longrightarrow \operatorname{(3)}]$ Assume that $E(A)=E(B)$ for some
$\mathfrak{E}$-objects $A, B$. Choose a construct
$\mathcal{X}$ having only one object $X$ and only one morphism $X\xrightarrow{{id}_{X}}X$.
Let $F_1 , F_{2}:\mathcal{X}\longrightarrow \mathfrak{E}$ be two
functors defined as follows: $F_1 (X)=A$ and $F_1 ({id}_{X})={id}_{A}$;
$F_{2}(X)=B$ and $F_{2}({id}_{X})={id}_{B}$. Since $E(A)=E(B)$, we know that
$[E\circ F_1](X)=[E\circ F_{2}](X)$ Further,
$[E\circ F_1]({id}_{X})=E({id}_{A})={Id}_{E(A)}$ and
$[E\circ F_{2}]({id}_{X})=E({id}_{B})={Id}_{E(B)}$. One more time due to the fact
that $E(A)=E(B)$, one has $[E\circ F_1]({id}_{X})=[E\circ F_{2}]({id}_{X})$ by
the uniqueness of the identity; so $F\circ F_1 = E \circ F_{2}$. Because
$E$ is monomorphism, it follows that $F_1 = F_{2}$. In particular,
$F_1 (X) = F_{2}(X)$, that is, $A=B$. Therefore, $E$ is injective on objects.

To show that $E$ is injective on morphisms, assume that $E(f)=E(g)$, where
$A\xrightarrow{f} B$ and $C\xrightarrow{g} D$ are $\mathfrak{E}$-morphisms.
Since we have shown that $E$ is injective on objects, we have $A=C$ and $B=D$.
Choose a construct $\mathcal{Y}$ having two objects $X, Y$, and
three morphisms ${id}_{X}, {id}_{Y}, h$, where $X\xrightarrow{h} Y$. Define
two functors $G_1 , G_{2}:\mathcal{Y}\longrightarrow \mathfrak{E}$
by the following rules.\\
\emph{Definition of} $G_1$. $G_1 (X)=A$, $G_1 (Y)=B$, $G_1 ({id}_{X})={id}_{A}$,
$G_1 ({id}_{Y})={id}_{B}$, $G_1 (h)=f$.\\
\emph{Definition of} $G_{2}$. $G_{2} (X)=A$, $G_{2} (Y)=B$, $G_{2}
({id}_{X})={id}_{A}$,
$G_{2} ({id}_{Y})={id}_{B}$, $G_{2} (h)=g$.

Since $E(f)=E(g)$, one has: $[E\circ G_1](X)= E(A)=[E\circ G_{2}](X)$;
$[E\circ G_{1}](Y)= E(B)=[E\circ G_{2}](Y)$; $[E\circ G_{1}]({id}_{X})
= E({id}_{A})={Id}_{E(A)}=[E\circ G_{2}]({id}_{X})$; $[E\circ G_{1}]({id}_{Y})=
E({id}_{B})={Id}_{E(B)}=
[E\circ G_{2}]({id}_{Y})$; $[E\circ G_{1}](h)= E(f)=E(g)=[E\circ G_{2}](h)$.
Therefore $[E\circ G_1]=[E\circ G_{2}]$. Since $E$
is monomorphism, we have $G_1 = G_{2}$. In particular,
$G_1 (h) = G_{2}(h)$, \emph{i.e.}, $f=g$. Thus, $E$ is
injective on morphisms, hence embedding.

\vspace{0.3 cm}

$[\operatorname{(3)}\Longrightarrow \operatorname{(1)}]$
Because $E:\mathfrak{E}\longrightarrow \mathcal{A}$ is an embedding, it is
both faithful and injective on objects. Then it can be considered,
up to isomorphisms, as an inclusion in its corresponding codomain. Therefore, it
suffices to show that
if $\mathfrak{E}$ is a sub-emergence of $\mathcal{A}$ then the inclusion
$I_{(\mathfrak{E}\hookrightarrow\mathcal{A})}:\mathfrak{E}
\hookrightarrow \mathcal{A}$
is a RHE (note that we changed the notation to
$I_{(\mathfrak{E}\hookrightarrow\mathcal{A})}$ instead of considering $E$ to denote
an inclusion although we did not change the notation for the domain $\mathfrak{E}$
nor for the codomain $\mathcal{A}$).

In order to do this, let $\mathfrak{E}$
be a sub-emergence of $\mathcal{A}$ with inclusion $I_{(\mathfrak{E}
\hookrightarrow \mathcal{A})}: \mathfrak{E}\hookrightarrow \mathcal{A}$. Assume
first that $\mathfrak{E}\subsetneqq \mathcal{A}$. We
will generate a construct $\mathcal{B}$ and two functors $F, G:
\mathcal{A}\longrightarrow \mathcal{B}$ such that $I_{(\mathfrak{E}
\hookrightarrow \mathcal{A})}: \mathfrak{E} \longrightarrow
\mathcal{A}$ is an equalizer emergence of $F$ and $G$.

Choose a construct $\mathcal{X}$ having at least two distinct
$\mathcal{X}$-objects $X, Y$, where $X$ and $Y$ are
isomorphic (the reason to assume such isomorphism will be clear along this
proof), \emph{i.e.}, there exist $\mathcal{X}$-morphisms
$X\xrightarrow{h}Y$ and $Y\xrightarrow{k} X$ such that
$k \circ h = {id}_{X}$ and $h \circ k = {id}_{Y}$.

Let $\mathcal{B}$ be the construct whose class of objects is given by
$\operatorname{Ob}(\mathcal{X})=\{ X, Y\}$, and the class of
morphisms is $\operatorname{Mor}(\mathcal{X})=\{{id}_{X}, {id}_{Y}, h, k \}$.
We next define the functors $F, G:\mathcal{A}\longrightarrow
\mathcal{B}$ as follows.\\
\emph{Definition of} $F$. For every $\mathcal{A}$-object $A$ we put
$F(A)=X$; for all
$\mathcal{A}$-morphism $f$ we take $F(f)={id}_{X}$.\\
\emph{Definition of} $G$. If $A$ is an
$\mathfrak{E}$-object, define $G(A)=X$, and if $A \in \operatorname{Ob}
(\mathcal{A})-\operatorname{Ob}(\mathfrak{E})$ we set $G(A)=Y$. For
every $f \in \operatorname{Mor}(\mathfrak{E})$, define $G(f)= {id}_{X}$;
for every $A\xrightarrow{f}B$ belonging to $\operatorname{Mor}(\mathcal{A})-
\operatorname{Mor}(\mathfrak{E})$ with $A \in \operatorname{Ob}
(\mathfrak{E})$ and $B \in \operatorname{Ob}
(\mathcal{A})- \operatorname{Ob}(\mathfrak{E})$, define $G(f)=h$;
for each morphism $A\xrightarrow{f}B$ where $A \in \operatorname{Ob}
(\mathcal{A})- \operatorname{Ob}(\mathfrak{E})$ and $B \in \operatorname{Ob}
(\mathfrak{E})$, define $G(f)=k$; if $A\xrightarrow{f}B$ is such that both
$A, B \in \operatorname{Ob}(\mathcal{A})- \operatorname{Ob}(\mathfrak{E})$
define $G(f)={id}_{Y}$.

It is easy to see that $F$ is a functor. We show next that $G$ is also
a functor. First, assume that $A \in \operatorname{Ob}
(\mathfrak{E})$; then $G({id}_{A})= {id}_{X}$. If $A \in \operatorname{Ob}
(\mathcal{A})- \operatorname{Ob}(\mathfrak{E})$, then $G({id}_{A})=
{id}_{Y}$. Then $G$ sends identities to identities. To prove that $G$
preserves compositions we have eight cases.
Below, we utilize the notation $X\longrightarrow X
\longrightarrow Y$ (instead of writing $X\xrightarrow{G(f)}
X\xrightarrow{G(g)} Y$) meaning the image of
$A\xrightarrow{f} B$ and $B\xrightarrow{g} C$, where $A$ is
an $\mathfrak{E}$-object,
$B$ is an $\mathfrak{E}$-object and $C$ is an object of
$\operatorname{Ob}(\mathcal{A})-\operatorname{Ob}
(\mathfrak{E})$.
\begin{itemize}
\item [ $\operatorname{Case \ 1)}$] $A, B, C \in \operatorname{Ob}
(\mathfrak{E})$. Let $A\xrightarrow{f} B$ and $B\xrightarrow{g} C$ be
$\mathfrak{E}$-morphisms. Thus
$G(g \circ f)={id}_{X}={id}_{X}\circ {id}_{X}=G(g)\circ G(f)$.

\item [ $\operatorname{Case \ 2)}$] $A, B, C \in \operatorname{Ob}
(\mathcal{A})-\operatorname{Ob}(\mathfrak{E})$. One has
$G(g \circ f)={id}_{Y}={id}_{Y}\circ {id}_{Y}=G(g)\circ G(f)$.

\item [ $\operatorname{Case \ 3)}$] $X\longrightarrow X \longrightarrow
Y$. $G(g \circ f)=h =h \circ {id}_{X}=G(g) \circ G(f)$.

\item [ $\operatorname{Case \ 4)}$] $X\longrightarrow Y \longrightarrow
X$. $G(g \circ f)={id}_{X} =k \circ h=G(g) \circ G(f)$.

\item [ $\operatorname{Case \ 5)}$] $Y\longrightarrow X \longrightarrow
X$. $G(g \circ f)=k ={id}_{X} \circ k=G(g) \circ G(f)$.

\item [ $\operatorname{Case \ 6)}$] $Y\longrightarrow Y \longrightarrow
X$. $G(g \circ f)=k = k \circ {id}_{Y} =G(g) \circ G(f)$.

\item [ $\operatorname{Case \ 7)}$] $Y\longrightarrow X \longrightarrow
Y$. $G(g \circ f)={id}_{Y} = h \circ k =G(g) \circ G(f)$.

\item [ $\operatorname{Case \ 8)}$] $X\longrightarrow Y \longrightarrow
Y$. $G(g \circ f)= h =  {id}_{Y} \circ h =G(g) \circ G(f)$.
\end{itemize}
Therefore, $G$ is a functor. It is easy to see that $F \circ I_{(\mathfrak{E}
\hookrightarrow \mathcal{A})}= G \circ I_{(\mathfrak{E} \hookrightarrow
\mathcal{A})}$. Assume that $\mathcal{C}$ is an emergence and let $T:\mathcal{C}
\longrightarrow \mathcal{A}$ be a functor such that $F \circ T = G \circ T$.
Hence, for each $C \in \operatorname{Ob}(\mathcal{C})$
one has $[F \circ T](C) = [G \circ T](C)=X$, which implies
that $T(C) \in \operatorname{Ob}(\mathfrak{E})$. Proceeding similarly
to morphisms, we conclude that for every $h \in
\operatorname{Mor}(\mathcal{C})$ we have $T(h)
\in \operatorname{Mor}(\mathfrak{E})$. Then, the unique functor
$\widehat{E}: \mathcal{C}\longrightarrow \mathfrak{E}$ such that
$T= I_{(\mathfrak{E} \hookrightarrow \mathcal{A})} \circ
\widehat{E}$ is the functor $\widehat{E}$ defined by
$\widehat{E}(C)=T(C)$ for all
$C \in \operatorname{Ob}(\mathcal{C})$, and
$\widehat{E}(f)={id}_{X}$ for all
$f \in \operatorname{Mor}(\mathcal{C})$.

If $\mathfrak{E}=\mathcal{A}$, then the identity functor
${Id}_{\mathcal{A}}:\mathfrak{E}\longrightarrow\mathcal{A}$
is an equalizer emergence of $F=G: \mathcal{A}\longrightarrow
\mathcal{B}$.
\begin{eqnarray*}
\begin{tikzpicture}
  \matrix (m) [matrix of math nodes,row sep=3em,column sep=3em,minimum width=2em]
  {
      \mathcal{C} &  & \\
      \mathfrak{E} &  \mathcal{A} & \mathcal{B}\\};
  \path[-stealth]
    (m-1-1) edge node [right] {$T$} (m-2-2)
    (m-1-1) edge node [left] {$\widehat{E}$} (m-2-1)
    (m-2-1) edge node [below] {$I_{(\mathfrak{E}\hookrightarrow \mathcal{A})}$} (m-2-2)
    (m-2-2) edge node [above] {$F$} (m-2-3)
    (m-2-2) edge node [below] {$G$} (m-2-3);

\end{tikzpicture}
\end{eqnarray*} The proof is complete.
\end{proof}

Equalizer emergences also satisfy the following result.

\begin{proposition}\label{equaemerg1}
Let $ {\mathcal{E}}_{\mathcal{A}}= ( \mathcal{A}, e_{\mathcal{A}}, U_{\mathcal{A}} )$,
$ {\mathcal{E}}_{\mathcal{B}}= ( \mathcal{B}, e_{\mathcal{B}}, U_{\mathcal{B}} )$
be emergences and $F, G:\mathcal{A}\longrightarrow \mathcal{B}$ be functors. Assume that
$E:\mathfrak{E}\longrightarrow \mathcal{A}$ is an equalizer emergence of $F$ and $G$.
Then the following conditions are equivalent:
\begin{itemize}
\item [ $\operatorname{(1)}$] $F=G$.

\item [ $\operatorname{(2)}$] $E$ is epimorphism of constructs.

\item [ $\operatorname{(3)}$] $E$ is an isomorphism of constructs.

\item [ $\operatorname{(4)}$] ${Id}_{\mathcal{A}}$ is
an equalizer emergence of $F$ and $G$.
\end{itemize}
\end{proposition}
\begin{proof}
The proof is the same as Proposition 7.54 in \cite{strecker:1990}.
Since it is not presented in \cite{strecker:1990}, we
exhibit it here for completeness.
It is clear that $\operatorname{(3)}\Longrightarrow \operatorname{(2)}$,
$\operatorname{(2)}\Longrightarrow \operatorname{(1)}$ and
$\operatorname{(1)}\Longrightarrow \operatorname{(4)}$ are true. We only show
that $\operatorname{(4)} \Longrightarrow \operatorname{(3)}$. Since
$E:\mathfrak{E}\longrightarrow \mathcal{A}$ is an
equalizer emergence of $F$ and $G$, then there exists a unique functor
$T_1: \mathcal{A}\longrightarrow
\mathfrak{E}$ such that $E \circ T_1 = Id_{\mathcal{A}}$; so, $E$ is a
retraction (\emph{i.e.}, $E$ has a right inverse). Because $E$ is a regular
monomorphism, it follows from Proposition~\ref{regmonomono} that
$E$ is a monomorphism. Since  $E$ is both retraction
and monomorphism, it implies that $E$ is isomorphism of constructs,
as required.
\begin{eqnarray*}
\begin{tikzpicture}
  \matrix (m) [matrix of math nodes,row sep=3em,column sep=3em,minimum width=2em]
  {
      \mathcal{A}&  \\
      \mathfrak{E}  &  \mathcal{A}\\};
  \path[-stealth]
    (m-1-1) edge node [left] {$T_{1}$} (m-2-1)
    (m-1-1) edge node [right] {$Id_{\mathcal{A}}$} (m-2-2)
    (m-2-1) edge node [below] {$E$} (m-2-2);
\end{tikzpicture}
\end{eqnarray*}
\end{proof}

The concept of equalizer emergence induces the idea in
which the generalized underlying functor are stabilized.

\begin{definition}\label{defstabili}
Let $ {\mathcal{E}}_{\mathcal{A}}= ( \mathcal{A}, e_{\mathcal{A}},
U_{\mathcal{A}} )$ and $ {\mathcal{E}}_{\mathcal{A}}^{*}=
( \mathcal{A}, e_{\mathcal{A}}, U_{\mathcal{A}}^{*} )$ be two emergences
defined over the same construct $\mathcal{A}$. We say that the functor
$E:\mathfrak{E}\longrightarrow \mathcal{A}$ stabilizes
$\mathcal{A}$ if it is
an equalizer emergence of $U_{\mathcal{A}}$ and $U_{\mathcal{A}}^{*}$.
\begin{eqnarray*}
\begin{tikzpicture}
  \matrix (m) [matrix of math nodes,row sep=3em,column sep=3em,minimum width=2em]
  {
    \mathfrak{E} & \mathcal{A} & \mathcal{S}\\};
  \path[-stealth]
    (m-1-1) edge node [above] {$E$} (m-1-2)
    (m-1-2) edge node [above] {$U_{\mathcal{A}}$} (m-1-3)
    (m-1-2) edge node [below] {$U_{\mathcal{A}}^{*}$} (m-1-3);
\end{tikzpicture}
\end{eqnarray*}
\end{definition}

\begin{remark}\label{biolostabi}
It follows from the definition of equalizer emergence that
$U_{\mathcal{A}}\circ E = U_{\mathcal{A}}^{*} \circ E$ and for
each functor $E^{*}:{\mathfrak{E}}^{*}\longrightarrow \mathcal{A}$
such that $U_{\mathcal{A}}\circ E^{*} = U_{\mathcal{A}}^{*}
\circ E^{*}$, there exists a unique functor
$\overline{E}:{\mathfrak{E}}^{*}\longrightarrow \mathfrak{E}$
such that ${\mathfrak{E}}^{*}= E \circ \overline{E}$. In terms
of biological systems these facts can be understood as a
regulation of the system in the sense that several diseases
can affect the equilibrium of some parts of it. The
uniqueness can be understood as change of techniques of
treatment.
\end{remark}

\begin{definition}\label{defemeequastrong}
Let $ {\mathcal{E}}_{\mathcal{A}}= ( \mathcal{A}, e_{\mathcal{A}},
U_{\mathcal{A}} )$, $ {\mathcal{E}}_{\mathcal{B}}=
( \mathcal{B}, e_{\mathcal{B}}, U_{\mathcal{B}} )$ be emergences. We
say that a functor $E:\mathfrak{E}\longrightarrow
\mathcal{A}$ is a strong equalizer emergence of $F$ and $G$
$F, G:\mathcal{A}\longrightarrow \mathcal{B}$ if $E$ is an equalizer
emergence of
$U_{\mathcal{B}}\circ F$ and $U_{\mathcal{B}}\circ G$.
\begin{eqnarray*}
\begin{tikzpicture}
  \matrix (m) [matrix of math nodes,row sep=3em,column sep=3em,minimum width=2em]
  {
    \mathfrak{E} &  \mathcal{A} &  \mathcal{B} & \mathcal{S}\\};
  \path[-stealth]
    (m-1-1) edge node [above] {$E$} (m-1-2)
    (m-1-2) edge node [above] {$F$} (m-1-3)
    (m-1-2) edge node [below] {$G$} (m-1-3)
    (m-1-3) edge node [above] {$U_{\mathcal{B}}$} (m-1-4);
\end{tikzpicture}
\end{eqnarray*}
\end{definition}

The following result gives conditions in which equalizer emergences are
strong equalizer emergences.

\begin{proposition}\label{homoembeequa}
Let $ {\mathcal{E}}_{\mathcal{A}}= ( \mathcal{A}, e_{\mathcal{A}},
U_{\mathcal{A}} )$, $ {\mathcal{E}}_{\mathcal{B}}=
( \mathcal{B}, e_{\mathcal{B}}, U_{\mathcal{B}} )$ be homomorphic emergences
and $F, G:\mathcal{A}\longrightarrow \mathcal{B}$ be homomorphisms.
Let $E:\mathfrak{E}\longrightarrow \mathcal{A}$ be an equalizer emergence
of $F$ and $G$. If $U_{\mathcal{B}}$ is embedding then $E$ is an
isomorphism and a strong equalizer emergence.
\end{proposition}
\begin{proof}
From hypotheses, $F$ and $G$ are homomorphisms and $U_{\mathcal{B}}$ is
embedding. From Proposition~\ref{homoembed}, it follows that $F=G$.
Applying Proposition~\ref{equaemerg1}, we have that $E$ is an
isomorphism of categories. Since $E$ is also a homomorphism, the
result follows.

We next prove that $E$ is an equalizer emergence of $U_{\mathcal{B}}
\circ F$ and $U_{\mathcal{B}}\circ G$.
We know that $[U_{\mathcal{B}} \circ F]\circ E = [U_{\mathcal{B}}
\circ G]\circ E$. Let ${\mathcal{E}}_{{\mathfrak{E}}^{*}}$ be an
emergence and consider the functor ${\mathfrak{E}}^{*}
\xrightarrow{E^{*}} \mathcal{A}$ such that $[U_{\mathcal{B}} \circ F]
\circ E^{*}=[U_{\mathcal{B}} \circ G]\circ E^{*}$. Because
$U_{\mathcal{B}}$ is embedding one has $F\circ E^{*}= G\circ E^{*}$.
Since $E$ is an equalizer emergence of $F$ and $G$, there exists a
unique functor $\widehat{E}:{\mathfrak{E}}^{*}
\longrightarrow \mathfrak{E}$ such that $E^{*}= E \circ \widehat{E}$.
Thus the equalities $E^{*}= E \circ \widehat{E}$ and $[U_{\mathcal{B}} \circ F]
\circ E^{*}=[U_{\mathcal{B}} \circ G]\circ E^{*}$ hold.
Suppose that $T: {\mathfrak{E}}^{*} \longrightarrow
\mathfrak{E}$ is a functor such that $E^{*}= E \circ T$ and
$[U_{\mathcal{B}} \circ F]
\circ E^{*}=[U_{\mathcal{B}} \circ G]\circ E^{*}$ are true. Since
$U_{\mathcal{B}}$ is embedding, it follows that
$F\circ E^{*}= G\circ E^{*}$. Again, from the fact that
$E$ is an equalizer emergence of $F$ and $G$, we have $\widehat{E}=T$.
Hence, there exists a unique functor $\widehat{E}:{\mathfrak{E}}^{*}
\longrightarrow \mathfrak{E}$ such that $E^{*}= E \circ \widehat{E}$
and $[U_{\mathcal{B}} \circ F]\circ E^{*}=[U_{\mathcal{B}}
\circ G]\circ E^{*}$, \emph{i.e.}, $E$ is
a strong equalizer emergence, as required.
\end{proof}
\begin{eqnarray*}
\begin{tikzpicture}
  \matrix (m) [matrix of math nodes,row sep=3em,column sep=3em,minimum width=2em]
  {
      {\mathfrak{E}}^{*}&  \\
      \mathfrak{E}  &  \mathcal{A}\\};
    \path[-stealth]
    (m-1-1) edge node [left] {$\widehat{E}$} (m-2-1)
    (m-1-1) edge node [right] {$E^{*}$} (m-2-2)
    (m-2-1) edge node [below] {$E$} (m-2-2);
\end{tikzpicture}
\end{eqnarray*}

The discussion made in Remark~\ref{biolostabi} induces a generalization
of the concept of equalizer emergences (Definition~\ref{defemeequa}).
This occurs because it is interesting to stabilize, in some cases, more
than two functions at the same time in a given biological system.

\begin{definition}\label{defemeequamulti}
Let $ {\mathcal{E}}_{\mathcal{A}}$, $ {\mathcal{E}}_{\mathcal{B}}$
and ${\mathcal{E}}_{\mathfrak{E}}$ be emergences. Assume that $F_1 ,
F_{2}, \ldots, F_{n}:\mathcal{A}\longrightarrow \mathcal{B}$ are functors.
The functor $E:\mathfrak{E}\longrightarrow \mathcal{A}$ is an
$n$-equalizer emergence of $F_1 , F_{2}, \ldots , F_{n}$ and if:\\
$\operatorname{(1)}$ $F_{i} \circ E = F_{j} \circ E$ for all $i, j
\in \{1, 2, \ldots , n\}$;\\
$\operatorname{(2)}$ for each emergence ${\mathcal{E}}_{{\mathfrak{E}}^{*}}$
and for each functor $E^{*}:{\mathfrak{E}}^{*}\longrightarrow \mathcal{A}$ such
that $F_{i} \circ E^{*} = F_{j} \circ E^{*}$ for all $i, j \in \{1, 2, \ldots , n\}$,
 there exists a unique functor $\overline{E}:
{\mathfrak{E}}^{*} \longrightarrow\mathfrak{E}$ such that $E^{*}= E
\circ \overline{E}$, and following diagram
\begin{eqnarray*}
\begin{tikzpicture}
  \matrix (m) [matrix of math nodes,row sep=3em,column sep=3em,minimum width=2em]
  {
      {\mathfrak{E}}^{*} &  & \\
      \mathfrak{E} &  \mathcal{A} & \mathcal{B}\\};
  \path[-stealth]
    (m-1-1) edge node [left] {$\overline{E}$} (m-2-1)
    (m-1-1) edge node [right] {$E^{*}$}(m-2-2)
    (m-2-1) edge node [below] {$E$} (m-2-2)
    (m-2-2) edge node [above] {$F_{i}, F_{j}$} (m-2-3);
\end{tikzpicture}
\end{eqnarray*} commutes for all $i, j \in \{1, 2, \ldots , n\}$.
\end{definition}

\begin{remark}
$\operatorname{(1)}$ Sometimes we say that an equalizer emergence is an
$2$-equalizer emergence according to Definition~\ref{defemeequamulti}.\\
$\operatorname{(2)}$ Analogously, it can be shown that $n$-equalizer emergences
are essentially unique.\\
$\operatorname{(3)}$ As mentioned above, sometimes biological systems
must regulate itself by means of several operations
performed at the same time. Thus, it is necessary to introduce the
concept of $n$-equalizers.\\
$\operatorname{(4)}$ The concept of $n$-strong equalizer emergence
is defined similarly as strong equalizer emergence. The latter
is also called  $2$-strong equalizer emergence.\\
$\operatorname{(5)}$ Definition~\ref{defstabili} can be also
naturally generalized to $n$ generalized underlying functors.
\end{remark}

The following two results are generalizations of
Theorem~\ref{existequaeme} and Proposition~\ref{equaemerg1},
respectively.

\begin{theorem}\label{existequaememulti}
Let $ {\mathcal{E}}_{\mathcal{A}}$ and $ {\mathcal{E}}_{\mathcal{B}}$
be emergences. Assume that
$F_{i}:\mathcal{A}\longrightarrow \mathcal{B}$ are functors for each $i= 1,
2, \ldots , n$. Then there exists an
equalizer $ \mathfrak{E} \xrightarrow{E} \mathcal{A}$ of $F_{1},
F_{2}, \ldots, F_{n}$.
\end{theorem}
\begin{proof}
Similar to that of Theorem~\ref{existequaeme}.
\end{proof}

\begin{proposition}\label{equaemerg1n}
Let $ {\mathcal{E}}_{\mathcal{A}}= ( \mathcal{A}, e_{\mathcal{A}},
U_{\mathcal{A}} )$, $ {\mathcal{E}}_{\mathcal{B}}= ( \mathcal{B},
e_{\mathcal{B}}, U_{\mathcal{B}} )$ be emergences and $F_1 , F_{2}, \ldots ,
F_{n}:\mathcal{A}\longrightarrow \mathcal{B}$ be
functors. Assume that $E:\mathfrak{E}\longrightarrow \mathcal{A}$ is an
$n$-equalizer emergence of $F_1 , F_{2}, \ldots , F_{n}$.
Then the following conditions are equivalent:
\begin{itemize}
\item [ $\operatorname{(1)}$] $F_{i}=F_{j}$ for all or all $i, j \in \{ = 1,
2, \ldots , n \}$;

\item [ $\operatorname{(2)}$] $E$ is epimorphism of constructs;

\item [ $\operatorname{(3)}$] $E$ is an isomorphism of constructs;

\item [ $\operatorname{(4)}$] $Id_{\mathcal{A}}$ is an $n$-equalizer emergence
of $F_{i}$, for all $i = 1, 2, \ldots , n$.
\end{itemize}
\end{proposition}
\begin{proof}
Similar to that of Proposition~\ref{equaemerg1}.
\end{proof}

We next define the concept of $n$-co-equalizer emergence directly,
without considering the particular case of two functors.

\begin{definition}\label{defemeco-equa}
Let $ {\mathcal{E}}_{\mathcal{A}}$, $ {\mathcal{E}}_{\mathcal{B}}$
and ${\mathcal{E}}_{\mathcal{C}}$ be emergences. Assume that
$F_1 , F_{2}, \ldots , F_{n}: \mathcal{A}\longrightarrow \mathcal{B}$ are
functors. The functor $C:\mathcal{B} \longrightarrow
\mathcal{C}$ is said to be an $n$-co-equalizer
emergence of  $F_1 , F_{2}, \ldots , F_{n}$ if:
\begin{itemize}
\item [ $\operatorname{(1)}$] $C \circ F_{i} = C \circ F_{j}$ for all $i,
j \in \{1, 2, \ldots , n\}$;

\item [ $\operatorname{(2)}$] for each emergence ${\mathcal{E}}_{{\mathcal{C}}^{*}}$
and for each functor $C^{*}:\mathcal{B}\longrightarrow {\mathcal{C}}^{*}$ such
that $ C^{*}\circ F_{i} = C^{*}\circ F_{j}$ for all $i, j \in \{1, 2, \ldots , n\}$,
there exists a unique functor $\overline{C}:
\mathcal{C} \longrightarrow {\mathcal{C}}^{*}$ such that $C^{*}= \overline{C}
\circ C$ such that the following diagram
\end{itemize}
\begin{eqnarray*}
\begin{tikzpicture}
  \matrix (m) [matrix of math nodes,row sep=3em,column sep=3em,minimum width=2em]
  {
      \mathcal{A} & \mathcal{B} & \mathcal{C} \\
      &  & {\mathcal{C}}^{*}\\};
  \path[-stealth]
    (m-1-1) edge node [above] {$F_{i}, F_{j}$} (m-1-2)
       (m-1-2) edge node [above] {$C$}(m-1-3)
    (m-1-2) edge node [below] {$C^{*}$}(m-2-3)
    (m-1-3) edge node [right] {$\overline{C}$}(m-2-3);
\end{tikzpicture}
\end{eqnarray*} commutes.
\end{definition}

\begin{remark}
We do not present results of co-equalizers since it is dual to
that of equalizers. In fact, from the Duality Principle we have:\\
$\operatorname{(1)}$ Co-equalizer emergences are essentially unique.\\
$\operatorname{(2)}$ For each result which is valid to equalizers
there exists an analogous for co-equalizers via dualiy.
\end{remark}

\subsection{Product and co-product of emergence}\label{subsourceeme}

In this subsection, we introduce the concept of \emph{source emergence}. Such definition
is similar to that of source in Category Theory.

\begin{definition}\label{emesource}
Let $ {\mathcal{E}}_{\mathcal{A}}= ( \mathcal{A}, e_{\mathcal{A}}, U_{\mathcal{A}} )$
be an emergence and assume that, for each
$i \in I$ where $I$ is a class, $ {\mathcal{E}}_{{\mathcal{A}}_{i}}= ( {\mathcal{A}}_{i},
e_{{\mathcal{A}}_{i}}, U_{{\mathcal{A}}_{i}} )$ are emergences. A source emergence
is a pair $({\mathcal{E}}_{\mathcal{A}}, {(F_{i})}_{i \in I})$, where
$F_{i}: \mathcal{A} \longrightarrow {\mathcal{A}}_{i}$ is a family of
functors with domain $\mathcal{A}$, indexed by $I$.
\end{definition}

The construct $\mathcal{A}$ is said to be the domain of the
source and the family of constructs ${({\mathcal{A}}_{i})}_{i \in I}$
is called the codomain of the source. We will denote a source emergence by
${({\mathcal{E}}_{\mathcal{A}}, F_{i})}_{I}$ or by
$ {({\mathcal{E}}_{\mathcal{A}} \xrightarrow
{F_{i}}{\mathcal{E}}_{{\mathcal{A}}_{i}})}_{I}$.

\begin{remark}
Although $F_{i}: \mathcal{A} \longrightarrow {\mathcal{A}}_{i}$ is a
family of functors from $\mathcal{A}$ to ${\mathcal{A}}_{i}$, by abuse
of notation, we denote a source emergence by $ {({\mathcal{E}}_{\mathcal{A}}
\xrightarrow {F_{i}}{\mathcal{E}}_{{\mathcal{A}}_{i}})}_{I}$.
\end{remark}

\begin{definition}\label{compsource}
If $ \mathbb{A}={({\mathcal{E}}_{\mathcal{A}} \xrightarrow
{F_{i}}{\mathcal{E}}_{{\mathcal{A}}_{i}})}_{I}$
is a source emergence and if, for each $i \in I$, $ {\mathbb{A}}_{i}=
{({\mathcal{E}}_{{\mathcal{A}}_{i}} \xrightarrow
{G_{ij}}{\mathcal{E}}_{{\mathcal{A}}_{ij}})}_{J_i}$
is a source emergence, then
$$ ({\mathbb{A}}_{i})\circ \mathbb{A}=
{({\mathcal{E}}_{\mathcal{A}} \xrightarrow
{G_{ij}\circ F_{i}}{\mathcal{E}}_{{\mathcal{A}}_{ij}})}_{i
\in I, j \in J_{i}}$$ is a
source emergence called the composite of $\mathbb{A}$ and of the
family ${({\mathbb{A}}_{i})}_{I}$.
\end{definition}

In this paper, given a source emergence $ \mathbb{A}={({\mathcal{E}}_{\mathcal{A}}
\xrightarrow {F_{i}}{\mathcal{E}}_{{\mathcal{A}}_{i}})}_{I}$ and a functor
$F: \mathcal{B}\longrightarrow \mathcal{A}$, we utilize the notation $ \mathbb{A} \circ F$
to denote $ \mathbb{A} \circ F={({\mathcal{E}}_{\mathcal{B}} \xrightarrow
{F_{i}\circ F}{\mathcal{E}}_{{\mathcal{A}}_{i}})}_{I}$.

\begin{definition}\label{monosource}
Let $\mathcal{B}$ be a category. A source emergence
$\mathbb{F}= {({\mathcal{E}}_{\mathcal{A}}, F_{i})}_{I}$ is said to be a
mono-source emergence if for any pair of functors $\mathcal{B}\xrightarrow{R}
\mathcal{A}$ and $\mathcal{B}\xrightarrow{S} \mathcal{A}$,
the equation $\mathbb{F} \circ R = \mathbb{F} \circ S$ (that is,
$ F_{i}\circ R= F_{i}\circ S$ for each $i \in I$), implies that $R=S$.
\end{definition}

\begin{definition}\label{prodemer}
A source emergence $\mathbb{P}= {({\mathcal{E}}_{\mathcal{P}} \xrightarrow
{P_{i}}{\mathcal{E}}_{{\mathcal{A}}_{i}})}_{I}$ is called a product
emergence if for every source emergence $\mathbb{A}={({\mathcal{E}}_{\mathcal{A}}
\xrightarrow {S_{i}}{\mathcal{E}}_{{\mathcal{A}}_{i}})}_{I}$
with the same codomain as $\mathcal{P}$, there exists a unique functor
(universal property) $F:\mathcal{A} \longrightarrow\mathcal{P}$ such that
$\mathbb{A}=\mathbb{P}\circ F$. In other words, the following diagram
\begin{eqnarray*}
\begin{tikzpicture}
  \matrix (m) [matrix of math nodes,row sep=3em,column sep=3em,minimum width=2em]
  {
     \mathcal{P} &  {\mathcal{A}}_{i}\\
     \mathcal{A} &  \\};
  \path[-stealth]
    (m-2-1) edge node [left] {$F$} (m-1-1)
    (m-2-1) edge node [below] {$S_{i}$} (m-1-2)
    (m-1-1) edge node [above] {$P_{i}$}(m-1-2);
\end{tikzpicture}
\end{eqnarray*}
commutes for each $i \in I$. A product emergence with codomain
${({\mathcal{A}}_{i})}_{I}$ is called a product of
${({\mathcal{A}}_{i})}_{I}$.
\end{definition}

The following result is a direct consequence of Definition~\ref{prodemer}.

\begin{proposition}\label{monoprodeme}
Every product emergence is a mono-source emergence.
\end{proposition}
\begin{proof}
The proof is similar to that of Proposition 10.21 in \cite{strecker:1990}.
\end{proof}

Product emergences are essentially unique, as states the following result.

\begin{proposition}\label{uniquepro}
For every family ${({\mathcal{A}}_{i})}_{i \in I}$ of constructs, product
emergences of ${({\mathcal{A}}_{i})}_{I}$ are essentially unique. More
precisely, if $\mathbb{P}= {({\mathcal{E}}_{\mathcal{P}} \xrightarrow {P_{i}}
{\mathcal{E}}_{{\mathcal{A}}_{i}})}_{I}$
is a product emergence of ${({\mathcal{A}}_{i})}_{I}$, then one has:
\begin{itemize}
\item [ $\operatorname{(1)}$] for each product emergence
$\mathbb{Q}= {({\mathcal{E}}_{\mathcal{Q}} \xrightarrow {Q_{i}}
{\mathcal{E}}_{{\mathcal{A}}_{i}})}_{I}$,
there exists an isomorphism $\mathcal{Q} \xrightarrow {H} \mathcal{P}$
with $\mathbb{Q}=\mathbb{P}\circ H$;

\item [ $\operatorname{(2)}$] for each isomorphism $\mathcal{A}
\xrightarrow{G} \mathcal{P}$
the source $\mathbb{P}\circ G$ is a product of ${({\mathcal{A}}_{i})}_{I}$.
\end{itemize}
\end{proposition}
\begin{proof}
It suffices to adapt the proof of Proposition 10.22 in \cite{strecker:1990}.
\end{proof}

Given a family ${({\mathcal{A}}_{i})}_{I}$ of constructs,
the unique product emergence of this family is
denoted by ${({\prod}{\mathcal{A}}_{i}, {\pi}_{j})}_{j \in I}$.
Since we are interested in investigating biological systems,
it is only necessary to guarantee the existence of products
of a finite number of emergences, as states the following
result.

\begin{theorem}\label{existprodemer}
Given any finite family ${\mathcal{A}}_{1}, {\mathcal{A}}_{2}, \ldots,
{\mathcal{A}}_{n}$ of constructs, there exists a product
emergence of ${({\mathcal{A}}_{i})}_{i=1}^{n}$. The product emergence is the
family of projections ${\pi}_{j}:{\prod}_{i=1}^{n}{\mathcal{A}}_{i}
\longrightarrow {\mathcal{A}}_{j}$, where
${\prod}_{i=1}^{n}{\mathcal{A}}_{i}$ is the cartesian product of
${({\mathcal{A}}_{i})}_{i=1}^{n}$.
\end{theorem}
\begin{proof}
Let $({\mathcal{E}}_{\mathcal{A}}, {(F_{i})}_{i=1}^{n})$ be a source
emergence, with functors $F_{i}: \mathcal{A} \longrightarrow
{\mathcal{A}}_{i}$. For each $\mathcal{A}$-objects $A, A^{*}$, and for each
$\mathcal{A}$-morphism
$A\xrightarrow{h} A^{*}$, define the functor $F: \mathcal{A}\longrightarrow
{\prod}_{i=1}^{n}{\mathcal{A}}_{i}$ such that $F(A)=(F_{1}(A), F_{2}(A),
\ldots , F_{n}(A))$ and
$F(h)= (F_{1}(h), F_{2}(h), \ldots , F_{n}(h))$, where for
each $i=1, 2,\ldots, n$, $F_{i}(h)$ is an ${\mathcal{A}}_{i}$-morphism
$F_{i}(A)\xrightarrow{F_{i}(h)} F_{i}(A^{*})$. From definition, it is
clear that ${\pi}_{i} \circ F = F_{i}$, for
each $i=1, 2,\ldots, n$, \emph{i.e.}, the diagram
\begin{eqnarray*}
\begin{tikzpicture}
  \matrix (m) [matrix of math nodes,row sep=3em,column sep=3em,minimum width=2em]
  {
     {\prod}_{i=1}^{n}{\mathcal{A}}_{i} & {\mathcal{A}}_{i} \\
     \mathcal{A} &  \\};
  \path[-stealth]
    (m-1-1) edge node [above] {${\pi}_{i}$} (m-1-2)
    (m-2-1) edge node [left] {$F$} (m-1-1)
    (m-2-1) edge node [below] {$F_{i}$}(m-1-2);
\end{tikzpicture}
\end{eqnarray*}
commutes for each $i$.

We will now show that $F$ is unique. To do this, assume that
$G: \mathcal{A}\longrightarrow {\prod}_{i=1}^{n}
{\mathcal{A}}_{i}$ is a functor  such that ${\pi}_{i} \circ G = F_{i}$, for
each $i=1, 2,\ldots, n$. For each $\mathcal{A}$-object $A$,
it follows from definition of $G$ that
$G(A)=(A_{1}, A_{2}, \ldots , A_{n})$ for some
${\mathcal{A}}_{i}$-objects $A_{i}$, $i=1, 2,\ldots, n$.
Thus $[{\pi}_{i} \circ G](A) = F_{i}(A)$ implies that $F_{i}(A)=A_{i}$
for each $i=1, 2,\ldots, n$; so $F(A)=G(A)$ for all
$\mathcal{A}$-object $A$. Hence $F=G$ on objects. Analogously,
for each $\mathcal{A}$-morphism $h$, $A\xrightarrow{h} A^{*}$, one has
$G(h)=(h_{1}, h_{2}, \ldots , h_{n})$ for some
${\mathcal{A}}_{i}$-morphisms $h_{i}$, $i=1, 2,\ldots, n$. Since
$[{\pi}_{i} \circ G](h) = F_{i}(h)$, it follows that
$F_{i}(h)= h_{i}$ for each $i=1, 2,\ldots, n$; hence $F(h)=G(h)$.
Therefore, $F=G$, and the proof is complete.
\end{proof}

We next introduce the concept of \emph{sink emergence}.

\begin{definition}\label{emesink}
Let $ {\mathcal{E}}_{\mathcal{A}}= ( \mathcal{A}, e_{\mathcal{A}},
U_{\mathcal{A}} )$ be an emergence and assume that, for each
$i \in I$ where $I$ is a class, $ {\mathcal{E}}_{{\mathcal{A}}_{i}}= ( {\mathcal{A}}_{i},
e_{{\mathcal{A}}_{i}}, U_{{\mathcal{A}}_{i}} )$ are emergences. A sink emergence
is a pair $({(F_{i})}_{i \in I}, {\mathcal{E}}_{\mathcal{A}})$, where
$F_{i}: {\mathcal{A}}_{i} \longrightarrow \mathcal{A}$ is a family of
functors with codomain $\mathcal{A}$, indexed by $I$.
\end{definition}

The family of constructs ${({\mathcal{A}}_{i})}_{i \in I}$ is called the
domain of the sink. We write
$ {({\mathcal{E}}_{\mathcal{A}}, F_{i})}_{I}$ or $ {({\mathcal{E}}_{{\mathcal{A}}_{i}}
\xrightarrow {F_{i}} {\mathcal{E}}_{\mathcal{A}})}_{I} $ to denote
a sink emergence.

\begin{definition}\label{coprodemer}
A sink emergence $\mathfrak{C} = {({\mathcal{E}}_{{\mathcal{A}}_{i}}
\xrightarrow {{\lambda}_{i}} {\mathcal{E}}_{\mathcal{C}})}_{I} $
is called a co-product emergence if for every sink emergence $\mathbb{A}=
{({\mathcal{E}}_{{\mathcal{A}}_{i}}
\xrightarrow {F_{i}} {\mathcal{E}}_{\mathcal{A}})}_{I}$, there exists
a unique functor $H:\mathcal{C} \longrightarrow\mathcal{A}$ such that
$\mathbb{A}=H \circ \mathfrak{C}$, that is, the following diagram
\begin{eqnarray*}
\begin{tikzpicture}
  \matrix (m) [matrix of math nodes,row sep=3em,column sep=3em,minimum width=2em]
  {
      \mathcal{C} & {\mathcal{A}}_{i} \\
      \mathcal{A} &  \\};
  \path[-stealth]
    (m-1-1) edge node [left] {$H$} (m-2-1)
    (m-1-2) edge node [below] {$F_{i}$} (m-2-1)
    (m-1-2) edge node [above] {${\lambda}_{i}$}(m-1-1);
\end{tikzpicture}
\end{eqnarray*}
commutes for each $i \in I$. A co-product emergence with domain
${({\mathcal{A}}_{i})}_{I}$ is called a co-product of
${({\mathcal{A}}_{i})}_{I}$,
\end{definition}

\begin{remark}
Analogously to product emergences, co-product emergences are also unique
up to isomorphisms. Since the proof of this result is similar to that of
Proposition~\ref{uniquepro}, we do not present it here.
For a family ${({\mathcal{A}}_{i})}_{I}$ of constructs, the unique
co-product emergence of this family is denoted by
${({\Upsilon}_{j}, {\coprod}{\mathcal{A}}_{i})}_{j \in I}$.
\end{remark}

The following result states that co-products of finite
families of constructs exist.

\begin{theorem}\label{existcopemer}
Let $ {\mathcal{A}}_{1}, \ldots , {\mathcal{A}}_{n}$ be a
finite family of constructs. Then there exists a co-product
emergence obtained from them.
\end{theorem}
\begin{proof}
Note that the constructs are not necessary disjoint; then, we must
generate a disjoint union of them in order to guarantee our
proof. To do this, choose any construct $\mathcal{X}$ having at least $n$ distinct
$\mathcal{X}$-objects. Let $X_1, X_{2}, \ldots , X_{n}$ be $n$
distinct $\mathcal{X}$-objects. We then form $n$ constructs
${\mathcal{X}}_1, {\mathcal{X}}_{2}, \ldots , {\mathcal{X}}_{n}$
whose unique ${\mathcal{X}}_{i}$-object is $X_{i}$ and
the unique ${\mathcal{X}}_{i}$-morphism is the identity
$X_{i}\xrightarrow{{id}_{{X}_{i}}} X_{i}$, for every $i=1, 2,
\ldots , n$. Let us consider $n$ disjoint constructs of the form
${\mathcal{A}}_{i}\times \{{\mathcal{X}}_{i}\}$, $i=1, 2,
\ldots , n$. The disjoint union ${\bigcup}_{i=1}^{n}({\mathcal{A}}_{i}\times
\{{\mathcal{X}}_{i}\})$ is also a construct and it is denoted by
${\biguplus}_{i=1}^{n}{\mathcal{A}}_{i}$. Note that each
${\mathcal{A}}_{i}\times \{{\mathcal{X}}_{i}\}$ is isomorphic
(as categories) to ${\mathcal{A}}_{i}$.

Let us now consider the construct ${\biguplus}_{i=1}^{n}
{\mathcal{A}}_{i}$ together with the $n$
functors ${\mathcal{A}}_{j}\xrightarrow{{\Upsilon}_{j}} {\biguplus}_{i=1}^{n}
{\mathcal{A}}_{i}$, defined by ${\Upsilon}_{j}(A)=(A,
X_{j})$ and ${\Upsilon}_{j} (f)=(f, {id}_{X_{j}})$,
where $A\in \operatorname{Ob}({\mathcal{A}}_{j})$ and $f\in
\operatorname{Mor}({\mathcal{A}}_{j})$ for every
$j=1, 2, \ldots, n$. We will show that the family ${(
{\mathcal{A}}_{j}\xrightarrow{{\Upsilon}_{j}} {\biguplus}_{i=1}^{n}
{\mathcal{A}}_{i})}_{j=1}^{n}$ is a co-product of
${({\mathcal{A}}_{i})}_{i=1}^{n}$.

Let us consider a sink emergence ${({\mathcal{A}}_{j}\xrightarrow{F_{j}}
\mathcal{B})}_{j=1}^{n}$. Define the functor
$T:{\biguplus}_{i=1}^{n}{\mathcal{A}}_{i}\longrightarrow \mathcal{B} $ as
follows: for each $i=1,2 , \ldots , n$, $T(A, X_{i})= F_{i}(A)$
for all ${\mathcal{A}}_{i}$-object $A$ and for all
${\mathcal{X}}_{i}$-object $X_{i}$; $T(f, {id}_{{X}_{i}})=
F_{i}(f)$ for all ${\mathcal{A}}_{i}$-morphism $f$
and for all ${\mathcal{X}}_{i}$-morphism (${X}_{i}$-identities)
${id}_{{X}_{i}}$. It is easy to see that $T$ is, in fact, a functor
and $F_{i}=T \circ {\Upsilon}_{i}$.

To prove the uniqueness, assume that $L:{\biguplus}_{i=1}^{n}
{\mathcal{A}}_{i}\longrightarrow \mathcal{B} $ is
a functor such that $F_{i}=L \circ {\Upsilon}_{i}$. Then, for each
${\mathcal{A}}_{i}$-object
$A$ and for (the unique) ${\mathcal{X}}_{i}$-object $X_{i}$, we have
$[L \circ {\Upsilon}_{i}](A)=
[T \circ {\Upsilon}_{i}](A)$, which implies that
$L(A, X_{i})=T(A, X_{i})$ for all $i=1, 2, \ldots , n$. Hence, $L=T$ on objects.
Analogously, for each $f \in \operatorname{Mor}
({\mathcal{A}}_{i})$ and for (the unique) ${\mathcal{X}}_{i}$-morphism
${id}_{X_{i}}$, one has $L(f, {id}_{X_{i}})=
T(f, {id}_{X_{i}})$ for all $i=1, 2, \ldots , n$, which implies that
$L=T$ on morphisms. Therefore, $L=T$. To finish the proof,
it suffices to take the construct ${\biguplus}_{i=1}^{n}
{\mathcal{A}}_{i}$ with its corresponding underlying functor.
\begin{eqnarray*}
\begin{tikzpicture}
  \matrix (m) [matrix of math nodes,row sep=3em,column sep=3em,minimum width=2em]
  {
      {\biguplus}_{i=1}^{n}{\mathcal{A}}_{i} & \mathcal{B}\\
      {\mathcal{A}}_{j} &  \\};

  \path[-stealth]
    (m-1-1) edge node [above] {$T$} (m-1-2)
    (m-1-1) edge node [below] {$L$} (m-1-2)
    (m-2-1) edge node [left] {${\Upsilon}_{j}$} (m-1-1)
    (m-2-1) edge node [right] {${F}_{j}$}(m-1-2);
\end{tikzpicture}
\end{eqnarray*} The proof is complete.
\end{proof}

\subsection{Limit and co-limit of emergences}

In this section we introduce the concepts of \emph{limits} and \emph{co-limits}
of emergences. In order to do this, we must define \emph{quasi-functor}.

\begin{definition}\label{quasifunctor}
Let $\mathbb{A}$ and $\mathbb{B}$ be two quasi-categories. Then a
quasi-functor $\mathrm{T}:\mathbb{A}\longrightarrow\mathbb{B}$ is a
function that assigns to each $\mathbb{A}$-object $\mathcal{A}$ a $\mathbb{B}$-object
$\mathrm{T}(\mathcal{A})$, and to each $\mathbb{A}$-morphism $\mathcal{A}
\xrightarrow{F}\mathcal{A}^{*}$ a $\mathbb{B}$-morphism
$\mathrm{T}(\mathcal{A})\xrightarrow{\mathrm{T}(F)}
\mathrm{T}({\mathcal{A}}^{*})$ such that $\mathrm{T}$ preserves
composition of morphisms and identities morphisms.
\end{definition}

Based on the previous definition, we can define \emph{diagram} for
emergence. We denote by ${\mathbb{C}}_{\mathbb{O}}$ the quasi-category
of all constructs.

\begin{definition}\label{diagramemer}
A diagram emergence is a quasi-functor
$\mathrm{D}:\mathbb{I} \longrightarrow{\mathbb{C}}_{\mathbb{O}}$
with codomain ${\mathbb{C}}_{\mathbb{O}}$. The domain $\mathbb{I}$ is
said to be the scheme of the diagram.
\end{definition}

\begin{definition}\label{naturemersource}
Let $\mathrm{D}:\mathbb{I}\longrightarrow{\mathbb{C}}_{\mathbb{O}}$
be a diagram emergence. A ${\mathbb{C}}_{\mathbb{O}}$-source
emergence ${(\mathcal{C}\xrightarrow{{F}_{i}}{\mathcal{D}}_{i})}_{i \in
\operatorname{Ob}(\mathbb{I})}$
is said to be natural for $\mathrm{D}$ if for each
$\mathbb{I}$-morphism
$i \xrightarrow{D} j$, the following triangle
\begin{eqnarray*}
\begin{tikzpicture}
  \matrix (m) [matrix of math nodes,row sep=3em,column sep=3em,minimum width=2em]
  {
     \mathcal{C} &  \\
     {\mathcal{D}}_{i} &  {\mathcal{D}}_{j} \\};
  \path[-stealth]
    (m-1-1) edge node [left] {${F}_{i}$} (m-2-1)
    (m-1-1) edge node [above] {${F}_{j}$} (m-2-2)
    (m-2-1) edge node [below] {$\mathrm{D}(D)$} (m-2-2);
\end{tikzpicture}
\end{eqnarray*} commutes.
\end{definition}

\begin{remark}
Note that an $\mathbb{I}$-morphism means a functor $D:i \longrightarrow
j$, where $i, j$ are categories belonging to the quasi-category
$\mathbb{I}$. Similarly, a ${\mathbb{C}}_{\mathbb{O}}$-morphism is a functor
$F:\mathcal{A} \longrightarrow \mathcal{B}$, where $\mathcal{A}$
and $\mathcal{B}$ are categories belonging to ${\mathbb{C}}_{\mathbb{O}}$,
\emph{i.e.}, they are constructs.
\end{remark}

Finally, we are able to define \emph{limit} for emergences.

\begin{definition}\label{deflimitemer}
Let $\mathrm{D}:\mathbb{I}\longrightarrow{\mathbb{C}}_{\mathbb{O}}$
be a diagram emergence. A limit emergence of $\mathrm{D}$ is a natural
source ${(\mathcal{L}\xrightarrow{{L}_{i}}{\mathcal{D}}_{i})}_{i\in\operatorname{Ob}
(\mathbb{I})}$ for $\mathrm{D}$ such that, for each
natural source ${(\mathcal{C}\xrightarrow{{F}_{i}}
{\mathcal{D}}_{i})}_{i \in \operatorname{Ob}(\mathbb{I})}$ for
$\mathrm{D}$, there exists a unique
${\mathbb{C}}_{\mathbb{O}}$-morphism $F:\mathcal{C}
\longrightarrow \mathcal{L}$ such that the following triangle
\begin{eqnarray*}
\begin{tikzpicture}
  \matrix (m) [matrix of math nodes,row sep=3em,column sep=3em,minimum width=2em]
  {
     \mathcal{L} &  \\
     \mathcal{C} &  {\mathcal{D}}_{i} \\};
  \path[-stealth]
    (m-2-1) edge node [left] {$F$} (m-1-1)
    (m-1-1) edge node [above] {${L}_{i}$} (m-2-2)
    (m-2-1) edge node [below] {$F_{i}$} (m-2-2);
\end{tikzpicture}
\end{eqnarray*} commutes for each $i \in \operatorname{Ob}(\mathbb{I})$.
\end{definition}

Limits emergences are unique up to isomorphisms.

\begin{proposition}\label{uniqlimitemer}
Let $\mathrm{D}:\mathbb{I}\longrightarrow{\mathbb{C}}_{\mathbb{O}}$
be a diagram emergence. If ${(\mathcal{L}\xrightarrow{{L}_{i}}
{\mathcal{D}}_{i})}_{i\in\operatorname{Ob} (\mathbb{I})}$ and ${(\mathcal{L^{*}}
\xrightarrow{{L}_{i}^{*}}{\mathcal{D}}_{i})}_{i\in\operatorname{Ob}
(\mathbb{I})}$ are limits emergences of $\mathrm{D}$, then there exists
an isomorphism $\mathcal{L^{*}}\xrightarrow{T} \mathcal{L}$ such that
${L}_{i}^{*}=L_{i}\circ T$ for all $i\in\operatorname{Ob} (\mathbb{I})$.
 Moreover, for each isomorphism
$\mathcal{C} \xrightarrow{K} \mathcal{L}$, it follows that
${(\mathcal{C} \xrightarrow{{L}_{i}\circ K}
{\mathcal{D}}_{i})}_{i\in\operatorname{Ob} (\mathbb{I})}$ is
a limit emergence of $\mathrm{D}$.
\end{proposition}
\begin{proof}
Let ${(\mathcal{L}\xrightarrow{{L}_{i}}
{\mathcal{D}}_{i})}_{i\in\operatorname{Ob} (\mathbb{I})}$ and ${(\mathcal{L^{*}}
\xrightarrow{{L}_{i}^{*}}{\mathcal{D}}_{i})}_{i\in\operatorname{Ob}
(\mathbb{I})}$ be limits of $\mathrm{D}$. Then there exist unique functors
$T: \mathcal{L^{*}}\longrightarrow \mathcal{L}$ and
$K:\mathcal{L}\longrightarrow\mathcal{L^{*}}$ such that
${L}_{i}^{*}=L_{i}\circ T$ and $L_{i} = {L}_{i}^{*} \circ K$ for all
$i\in\operatorname{Ob} (\mathbb{I})$. Hence $L_{i} = L_{i} \circ
(T \circ K)$ and ${L}_{i}^{*} ={L}_{i}^{*} \circ (K \circ T)$. It
is easy to see that limits are mono-sources; so,
$T \circ K = {Id}_{\mathcal{L}}$ and $K \circ T =
{Id}_{{\mathcal{L}}^{*}}$. Therefore, $T$ is an isomorphism.
The second part is similar to the first one.
\end{proof}

\begin{definition}\label{naturemersink}
Let $\mathrm{D}:\mathbb{I}\longrightarrow{\mathbb{C}}_{\mathbb{O}}$
be a diagram emergence. A ${\mathbb{C}}_{\mathbb{O}}$-sink emergence
${({\mathcal{D}}_{i}\xrightarrow{{F}_{i}}\mathcal{C})}_{i \in
\operatorname{Ob}(\mathbb{I})}$
is called natural for $\mathrm{D}$ if for each $\mathbb{I}$-morphism
$i \xrightarrow{D} j$, the following triangle commutes.
\begin{eqnarray*}
\begin{tikzpicture}
  \matrix (m) [matrix of math nodes,row sep=3em,column sep=3em,minimum width=2em]
  {
     {\mathcal{D}}_{i} & {\mathcal{D}}_{j}\\
          & \mathcal{C}\\};
  \path[-stealth]
    (m-1-1) edge node [above] {$\mathrm{D}(D)$} (m-1-2)
    (m-1-1) edge node [left] {${F}_{i}$} (m-2-2)
    (m-1-2) edge node [right] {$F_{j}$} (m-2-2);
\end{tikzpicture}
\end{eqnarray*}
\end{definition}

Analogously to limit emergence, we can define \emph{co-limit} for emergences.

\begin{definition}\label{defcolimitemer}
Let $\mathrm{D}:\mathbb{I}\longrightarrow{\mathbb{C}}_{\mathbb{O}}$
be a diagram emergence. A co-limit emergence of $\mathrm{D}$ is a natural
sink ${({\mathcal{D}}_{i}\xrightarrow{{F}_{i}}\overline{\mathcal{L}})}_{i \in
\operatorname{Ob}(\mathbb{I})}$ for $\mathrm{D}$ such that, for
each natural sink ${({\mathcal{D}}_{i}
\xrightarrow{{G}_{i}}{\mathcal{C}})}_{i \in
\operatorname{Ob}(\mathbb{I})}$ for $\mathrm{D}$, there exists a unique
functor (universal property) $R:\overline{\mathcal{L}}
\longrightarrow \mathcal{C}$ such that
$G_{i} = R \circ F_{i}$ for all $i \in \operatorname{Ob}(\mathbb{I})$.
\end{definition}

Limits and co-limits of emergences are very interesting in our context,
since we can control a given biological situation by means of
controlling the schemes of the corresponding diagram. Because organisms have
a finite number of cells, tissues, etc, we can consider that the
corresponding schemes are finite. Therefore, we can control their
internal morphisms in order to imply some type of treatment. Moreover,
according to the ideas of Rosen (see Theorem~\cite[Theorem 3]{rosen:1958B}
in Section~\ref{sec5}) limits can provide "irreducible" ABDs. This is
interesting due to the fact that one can work with minimal collections
of objects and their corresponding functors.

\begin{remark}
$\operatorname{(1)}$ Similarly to limit emergences, co-limits emergences
are also unique up to isomorphisms.\\
$\operatorname{(2)}$ From the Duality Principle, co-limits emergences are
extremal epi-sinks emergences.\\
$\operatorname{(3)}$ Co-products emergences are particular cases of
co-limits emergences (co-limits of diagrams with discrete schemes).
\end{remark}

\subsection{Pullback and pushout emergence}\label{subpullbackeme}

In this subsection we deal with the introduction of \emph{pullback}
(pushout) of emergences. The following definition describes what
we mean by pullback in this context.

\begin{definition}\label{defpullbackemer}
Let $ {\mathcal{E}}_{\mathcal{A}}= ( \mathcal{A}, e_{\mathcal{A}},
U_{\mathcal{A}} )$, $ {\mathcal{E}}_{\mathcal{B}}=
( \mathcal{B}, e_{\mathcal{B}}, U_{\mathcal{B}} )$ and $ {\mathcal{E}}_{\mathcal{P}}=
( \mathcal{P}, e_{\mathcal{P}}, U_{\mathcal{P}} )$ be emergences.
A square
\begin{eqnarray*}
\begin{tikzpicture}
  \matrix (m) [matrix of math nodes,row sep=3em,column sep=3em,minimum width=2em]
  {
      \mathcal{P} & \mathcal{B} \\
      \mathcal{A} &  \mathcal{S}\\};
  \path[-stealth]
    (m-1-1) edge node [above] {$\overline{F}$} (m-1-2)
    (m-1-1) edge node [left] {$\overline{G}$}(m-2-1)
    (m-2-1) edge node [below] {$U_{\mathcal{A}}$}(m-2-2)
    (m-1-2) edge node [right] {$U_{\mathcal{B}}$}(m-2-2);
\end{tikzpicture}
\end{eqnarray*}
is said to be a pullback square emergence if it commutes and for all
commuting square
\begin{eqnarray*}
\begin{tikzpicture}
  \matrix (m) [matrix of math nodes,row sep=3em,column sep=3em,minimum width=2em]
  {
      \widehat{\mathcal{P}} & \mathcal{B} \\
      \mathcal{A} &  \mathcal{S}\\};
  \path[-stealth]
    (m-1-1) edge node [above] {$\widehat{F}$} (m-1-2)
    (m-1-1) edge node [left] {$\widehat{G}$}(m-2-1)
    (m-2-1) edge node [below] {$U_{\mathcal{A}}$}(m-2-2)
    (m-1-2) edge node [right] {$U_{\mathcal{B}}$}(m-2-2);
\end{tikzpicture}
\end{eqnarray*}
there exists a unique functor $L:\widehat{\mathcal{P}}
\longrightarrow\mathcal{P}$ such that the following diagram
\begin{eqnarray*}
\begin{tikzpicture}
  \matrix (m) [matrix of math nodes,row sep=3em,column sep=3em,minimum width=2em]
  {
     \widehat{\mathcal{P}} & &\\
     & \mathcal{P} & \mathcal{B} \\
     & \mathcal{A} &  \mathcal{S}\\};
  \path[-stealth]
    (m-1-1) edge node [above] {$L$} (m-2-2)
    (m-1-1) edge node [left] {$\widehat{G}$} (m-3-2)
    (m-1-1) edge node [above] {$\widehat{F}$} (m-2-3)
    (m-2-2) edge node [below] {$\overline{F}$} (m-2-3)
    (m-2-2) edge node [right] {$\overline{G}$}(m-3-2)
    (m-3-2) edge node [below] {$U_{\mathcal{A}}$}(m-3-3)
    (m-2-3) edge node [right] {$U_{\mathcal{B}}$}(m-3-3);
\end{tikzpicture}
\end{eqnarray*} commutes.
The $2$-source emergence $\mathcal{A} \xleftarrow{\overline{G}}\mathcal{P}
\xrightarrow{\overline{F}}  \mathcal{B}$ is called a
pullback emergence of the $2$-sink $\mathcal{A}\xrightarrow{U_{\mathcal{A}}} \mathcal{C}
 \xleftarrow{U_{\mathcal{B}}} \mathcal{B}$ and $\overline{F}$ is
said to be a pullback of $U_{\mathcal{A}}$ along $U_{\mathcal{B}}$.
\end{definition}

\begin{proposition}\label{pulbaemeruniq}
For fixed emergences $ {\mathcal{E}}_{\mathcal{A}}$
and ${\mathcal{E}}_{\mathcal{B}}$, pullback emergences are essentially unique.
\end{proposition}
\begin{proof}
Similar to that of Proposition~\ref{uniqlimitemer}.
\end{proof}

There exists an interesting way to construct pullbacks (see for instance
Proposition $11.11$ in \cite{strecker:1990}). In our case,
we must define pullback between emergences, that is, we are working in quasi-categories.
Since our objects are constructs and our morphisms are functors, we
have the following construction.

\begin{theorem}(Pullback emergence construction)\label{pulluniemerco}
Let $ {\mathcal{E}}_{\mathcal{A}}= ( \mathcal{A}, e_{\mathcal{A}},
U_{\mathcal{A}} )$, $ {\mathcal{E}}_{\mathcal{B}}=
( \mathcal{B}, e_{\mathcal{B}}, U_{\mathcal{B}} )$ be emergences.
Then there exists an emergence ${\mathcal{E}}_{\mathcal{P}}=
( \mathcal{P}, e_{\mathcal{P}},
U_{\mathcal{P}} )$ and functors
$\mathcal{P} \xrightarrow{\overline{G}}\mathcal{A}$ and
$\mathcal{P} \xrightarrow{\overline{F}}\mathcal{B}$ such that the
square
\begin{eqnarray*}
\begin{tikzpicture}
  \matrix (m) [matrix of math nodes,row sep=3em,column sep=3em,minimum width=2em]
  {
      \mathcal{P} & \mathcal{B} \\
      \mathcal{A} &  \mathcal{S}\\};
  \path[-stealth]
    (m-1-1) edge node [above] {$\overline{F}$} (m-1-2)
    (m-1-1) edge node [left] {$\overline{G}$}(m-2-1)
    (m-2-1) edge node [below] {$U_{\mathcal{A}}$}(m-2-2)
    (m-1-2) edge node [right] {$U_{\mathcal{B}}$}(m-2-2);
\end{tikzpicture}
\end{eqnarray*}
is a pullback square.
\end{theorem}
\begin{proof}
Define $\mathcal{P}$ as follows:
$\operatorname{Ob}(\mathcal{P})=\{ (A, B) \in
\operatorname{Ob}(\mathcal{A})\times \operatorname{Ob}(\mathcal{B}) \ | \
U_{\mathcal{A}}(A)=U_{\mathcal{B}}(B)\} $ and
$\operatorname{Mor}(\mathcal{P})= \{(f, g) \in \operatorname{Mor}
(\mathcal{A})\times \operatorname{Mor}(\mathcal{B}) \ | \
U_{\mathcal{A}}(f)=U_{\mathcal{B}}(g)\}$.
Here, we write the equality $U_{\mathcal{A}}(A)=
U_{\mathcal{B}}(B)$ for the GU functors $U_{\mathcal{A}}$ and
$U_{\mathcal{B}}$ considering a more general
setting. More precisely, these GU functors are considered, in fact,
generalized semi-underlying functors (see
Definition~\ref{defnew9semi} of Subsection~\ref{subsemieme}).

We will see that $\operatorname{Mor}(\mathcal{P})$ is well defined.
In fact, for every $(A, B) \in \operatorname{Ob}(\mathcal{P})$, the pair
$({id}_{A}, {id}_{B}) \in \operatorname{Mor}(\mathcal{P})$,
because $U_{\mathcal{A}}({id}_{A})=U_{\mathcal{B}}({id}_{B})$.
Furthermore, consider that $(f_{1}, g_{1}), (f_{2}, g_{2})$
belong to $\operatorname{Mor}(\mathcal{P})$ such that the
componentwise composites are defined, \emph{i.e.},
$A_{1}\xrightarrow{f_{1}} A_{2}$, $ A_{2}\xrightarrow{f_{2}} A_{3}$ and
$B_{1}\xrightarrow{g_{1}} B_{2}$, $ B_{2}\xrightarrow{g_{2}} B_{3}$. Since
$U_{\mathcal{A}}$ and $U_{\mathcal{B}}$ are functors, we have
\begin{eqnarray*}
U_{\mathcal{A}}(f_{2}\circ f_{1})&=& U_{\mathcal{A}}(f_{2})\circ
U_{\mathcal{A}}(f_{1})\\&=&
U_{\mathcal{B}}(g_{2})\circ U_{\mathcal{B}}(g_{1})\\&=&
U_{\mathcal{B}}(g_{2}\circ g_{1})
\end{eqnarray*}
Thus, the composite $(f_{2}\circ f_{1},  g_{2}\circ g_{1})$ also belongs to
$\operatorname{Mor}(\mathcal{P})$.

Let $\overline{F}={[{\Pi}_{\mathcal{B}}]}_{{|}_{\mathcal{P}}}$
and $\overline{G}={[{\Pi}_{\mathcal{A}}]}_{{|}_{\mathcal{P}}}$.
From construction, the square
\begin{eqnarray}\label{squareI}
\begin{tikzpicture}
  \matrix (m) [matrix of math nodes,row sep=3em,column sep=3em,minimum width=2em]
  {
      \mathcal{P} & \mathcal{B} \\
      \mathcal{A} &  \mathcal{S}\\};
  \path[-stealth]
    (m-1-1) edge node [above] {${[{\Pi}_{\mathcal{B}}]}_{{|}_{\mathcal{P}}}$} (m-1-2)
    (m-1-1) edge node [left] {${[{\Pi}_{\mathcal{A}}]}_{{|}_{\mathcal{P}}}$}(m-2-1)
    (m-2-1) edge node [below] {$U_{\mathcal{A}}$}(m-2-2)
    (m-1-2) edge node [right] {$U_{\mathcal{B}}$}(m-2-2);
\end{tikzpicture}
\end{eqnarray} commutes. Assume that the square
\begin{eqnarray*}
\begin{tikzpicture}
  \matrix (m) [matrix of math nodes,row sep=3em,column sep=3em,minimum width=2em]
  {
      \widehat{\mathcal{P}} & \mathcal{B} \\
      \mathcal{A} &  \mathcal{S}\\};
  \path[-stealth]
    (m-1-1) edge node [above] {$\widehat{F}$} (m-1-2)
    (m-1-1) edge node [left] {$\widehat{G}$}(m-2-1)
    (m-2-1) edge node [below] {$U_{\mathcal{A}}$}(m-2-2)
    (m-1-2) edge node [right] {$U_{\mathcal{B}}$}(m-2-2);
\end{tikzpicture}
\end{eqnarray*} commutes. Thus $U_{\mathcal{B}}\circ \widehat{F}=
U_{\mathcal{A}}\circ \widehat{G}$. Take a $\widehat{\mathcal{P}}$-object
$X$; one has $\left[ U_{\mathcal{B}}\circ \widehat{F}\right](X)=
\left[U_{\mathcal{A}}\circ \widehat{G}\right](X)$. Hence,
$(\widehat{G}(X) , \widehat{F}(X))$ belongs to
$\operatorname{Ob}(\mathcal{P})$. Analogously, given a
$\widehat{\mathcal{P}}$-morphism ${\widehat{P}}_{1}
\xrightarrow{\widehat{p}} {\widehat{P}}_{2}$ we have
$\left[U_{\mathcal{B}}\circ \widehat{F}\right](\widehat{p})=
\left[U_{\mathcal{A}}\circ \widehat{G}\right](\widehat{p})$. Thus,
$\left(\widehat{G}(\widehat{p}) , \widehat{F}(\widehat{p})\right)
\in \operatorname{Mor}(\mathcal{P})$.

Let us define $L:\widehat{\mathcal{P}}\longrightarrow\mathcal{P}$ as
follows: for each $\widehat{\mathcal{P}}$-object $X$,
$L(X)= \left(\widehat{G}(X) , \widehat{F}(X)\right)$, and for each
$\widehat{\mathcal{P}}$-morphism $\widehat{p}$,
$L(\widehat{p})= \left(\widehat{G}(\widehat{p}),
\widehat{F}(\widehat{p})\right)$.
We know that $L$ is well defined,
${[{\Pi}_{\mathcal{A}}]}_{{|}_{\mathcal{P}}}\circ L = \widehat{G}$ and
${[{\Pi}_{\mathcal{B}}]}_{{|}_{\mathcal{P}}}\circ L = \widehat{F}$.
It is easy to see that $L$ is a functor.

To show that $L$ is unique, assume that $L^{*}$ is a functor
$L^{*}:\widehat{\mathcal{P}}\longrightarrow\mathcal{P}$ such that
${[{\Pi}_{\mathcal{A}}]}_{{|}_{\mathcal{P}}}\circ L^{*}
= \widehat{G}$ and
${[{\Pi}_{\mathcal{B}}]}_{{|}_{\mathcal{P}}}\circ L^{*} = \widehat{F}$.
Then ${[{\Pi}_{\mathcal{A}}]}_{{|}_{\mathcal{P}}}\circ L^{*}
={[{\Pi}_{\mathcal{A}}]}_{{|}_{\mathcal{P}}}\circ L$ and
${[{\Pi}_{\mathcal{B}}]}_{{|}_{\mathcal{P}}}\circ L^{*}=
{[{\Pi}_{\mathcal{B}}]}_{{|}_{\mathcal{P}}}\circ L$.
Let $X$ be a $\widehat{\mathcal{P}}$-object and consider that
$L^{*}(X)=(L_{1}^{*}(X), L_{2}^{*}(X))$, with
$(L_{1}^{*}(X), L_{2}^{*}(X)) \in \operatorname{Ob} (\mathcal{P})$. Thus,
\begin{eqnarray*}
\left[{[{\Pi}_{\mathcal{A}}]}_{{|}_{\mathcal{P}}}\circ L^{*}\right](X) &=&
\left[{[{\Pi}_{\mathcal{A}}]}_{{|}_{\mathcal{P}}}\circ L\right](X)\\
&\Longrightarrow &
\left[{[{\Pi}_{\mathcal{A}}]}_{{|}_{\mathcal{P}}}\right](L_{1}^{*}(X), L_{2}^{*}(X))=
\left[{[{\Pi}_{\mathcal{A}}]}_{{|}_{\mathcal{P}}}\right](\widehat{G}(X) ,
\widehat{F}(X))\\&\Longrightarrow &
L_{1}^{*}(X)=\widehat{G}(X)
\end{eqnarray*}
and
\begin{eqnarray*}
\left[{[{\Pi}_{\mathcal{B}}]}_{{|}_{\mathcal{P}}}\circ L^{*}\right](X)&=&
\left[{[{\Pi}_{\mathcal{B}}]}_{{|}_{\mathcal{P}}}\circ L\right](X)\\&\Longrightarrow &
\left[{[{\Pi}_{\mathcal{B}}]}_{{|}_{\mathcal{P}}}\right](L_{1}^{*}(X), L_{2}^{*}(X))=
\left[{[{\Pi}_{\mathcal{B}}]}_{{|}_{\mathcal{P}}}\right]
(\widehat{G}(X) , \widehat{F}(X))\\&\Longrightarrow&
L_{2}^{*}(X)=\widehat{F}(X)
\end{eqnarray*}
Hence, it follows that $L_{1}^{*}=\widehat{G}$ and $L_{2}^{*}=\widehat{F}$, so
$L^{*}=L$ on $\widehat{\mathcal{P}}$-objects. Proceeding similarly to
morphisms, one has $L^{*}=L$ on $\widehat{\mathcal{P}}$-morphism; so, $L^{*}=L$.
Therefore, the Square~\ref{squareI} is a pullback, as required. To complete
the proof, it suffices to define the GU as being the usual underlying
functor functor of $\mathcal{P}$.
\end{proof}

Pullback are important to describe emergence phenomena because
we can extract ``new information from old ones". More precisely,
if we have two correlated biological systems $\mathcal{A}$ and $\mathcal{B}$
for example, we can estimate new correlations
between them, since pullbacks have at least the same
information than $\mathcal{A}$ and $\mathcal{B}$
because pullbacks are constructed based on products.

Analogously to pullback, we can also define \emph{pushout} of emergences.

\begin{definition}\label{defpushoutemer}
Let $ {\mathcal{E}}_{\mathcal{A}}= ( \mathcal{A}, e_{\mathcal{A}},
U_{\mathcal{A}} )$, $ {\mathcal{E}}_{\mathcal{B}}=
( \mathcal{B}, e_{\mathcal{B}}, U_{\mathcal{B}} )$ and $ {\mathcal{E}}_{\mathcal{P}}=
( \mathcal{C}, e_{\mathcal{C}}, U_{\mathcal{C}} )$ be emergences
 A square
\begin{eqnarray*}
\begin{tikzpicture}
  \matrix (m) [matrix of math nodes,row sep=3em,column sep=3em,minimum width=2em]
  {
      \mathcal{C} & \mathcal{B} \\
      \mathcal{A} &  \mathcal{S}\\};
  \path[-stealth]
    (m-1-1) edge node [above] {$F$} (m-1-2)
    (m-1-1) edge node [left] {$G$}(m-2-1)
    (m-2-1) edge node [below] {$U_{\mathcal{A}}$}(m-2-2)
    (m-1-2) edge node [right] {$U_{\mathcal{B}}$}(m-2-2);
\end{tikzpicture}
\end{eqnarray*}
is said to be a pushout emergence square if it commutes and for all
commuting square
\begin{eqnarray*}
\begin{tikzpicture}
  \matrix (m) [matrix of math nodes,row sep=3em,column sep=3em,minimum width=2em]
  {
      \mathcal{C} & \mathcal{B} \\
      \mathcal{A} &  \mathcal{P}\\};
  \path[-stealth]
    (m-1-1) edge node [above] {$F$} (m-1-2)
    (m-1-1) edge node [left] {$G$}(m-2-1)
    (m-2-1) edge node [below] {$\overline{F}$}(m-2-2)
    (m-1-2) edge node [right] {$\overline{G}$}(m-2-2);
\end{tikzpicture}
\end{eqnarray*}
there exists a unique functor $L:\mathcal{S}
\longrightarrow\mathcal{P}$ such that the following diagram
\begin{eqnarray*}
\begin{tikzpicture}
  \matrix (m) [matrix of math nodes,row sep=3em,column sep=3em,minimum width=2em]
  {
      \mathcal{C} & \mathcal{B} &\\
      \mathcal{A} & \mathcal{S} &\\
     & & \mathcal{P} \\};
  \path[-stealth]
    (m-1-1) edge node [above] {$F$} (m-1-2)
    (m-1-1) edge node [left] {$G$} (m-2-1)
    (m-2-1) edge node [above] {$U_{\mathcal{A}}$} (m-2-2)
    (m-2-1) edge node [right] {$\overline{F}$} (m-3-3)
    (m-2-2) edge node [right] {$K$}(m-3-3)
    (m-1-2) edge node [left] {$U_{\mathcal{B}}$}(m-2-2)
    (m-1-2) edge node [right] {$\overline{G}$}(m-3-3);
\end{tikzpicture}
\end{eqnarray*} commutes.
The $2$-sink emergence $((U_{\mathcal{B}}, U_{\mathcal{A}}), \mathcal{S})$
is called a pushout emergence of the $2$-source $(\mathcal{C}, (F, G))$,
and $U_{\mathcal{A}}$ is called pushout of $F$ along $G$.
\end{definition}

\begin{definition}\label{defpulation}
If a square is both pullback emergence and pushout emergence square then
it is called pulation emergence square.
\end{definition}

\begin{remark}
Although we have presented the definition of pushout emergence
(see Definition~\ref{defpushoutemer}) note that it is not
much interesting in our case. This can be understood as a biological system
that is ``redirected out" from its ``ground set" $\mathcal{S}$.
More precisely, we ``lose control" of the biological situation.
Due to this fact, we do not present results on pushouts
nor pulations.
\end{remark}

\subsection{Partial and relative emergence}\label{subpartialeme}

In this subsection, we introduce the concepts of \emph{partial}
and \emph{relative} emergences.
Partial emergences are important to
describe constructs which are emergent when compared with other
constructs, maintaining however their corresponding GU functors
(See Definition~\ref{partemer}). In the case of relative emergences,
we do not consider their GU functors (see Definition~\ref{relaemer}).
In the latter case, our interest is to compare ``purely" the
constructs.

\begin{definition}\label{partemer}
Let $\mathcal{A}$ and $\mathcal{B}$ be constructs. A partial
emergence ${\mathcal{E}}^{p}_{\mathcal{A}\mathcal{B}}$
is an ordered triple $ {\mathcal{E}}_{\mathcal{A}\mathcal{B}}=
( \mathcal{A}, \mathcal{B}, V_{\mathcal{A}\mathcal{B}})$,
where $V_{\mathcal{A}\mathcal{B}}:\mathcal{A}\longrightarrow
\mathcal{B}$ is a functor,
$|e_{\mathcal{A}}| > |e_{\mathcal{B}}|$, and
the following diagram
\begin{eqnarray*}
\begin{tikzpicture}
  \matrix (m) [matrix of math nodes,row sep=3em,column sep=3em,minimum width=2em]
  {
     \mathcal{B} & \mathcal{S} \\
     \mathcal{A} &  \\};
  \path[-stealth]
    (m-2-1) edge node [left] {$V_{\mathcal{A}\mathcal{B}}$} (m-1-1)
            edge node [below] {$U_{\mathcal{A}}$} (m-1-2)
    (m-1-1) edge node [above] {$U_{\mathcal{B}}$}(m-1-2);
\end{tikzpicture}
\end{eqnarray*}
commutes. The functor $V_{\mathcal{A}\mathcal{B}}$ is
called GU partial functor. The degree of the partial emergence
$\partial[{\mathcal{E}}^{p}_{\mathcal{A}\mathcal{B}}]$ is defined as
$\partial[{\mathcal{E}}_{\mathcal{A}\mathcal{B}}] = |e_{\mathcal{A}}|
- |e_{\mathcal{B}}|$.
\end{definition}

\begin{remark}
$\operatorname{ (1)}$ In the case in which $\mathcal{B}= \mathcal{S}$
($\mathcal{B}$ is not a construct in this case),
the partial emergence is an emergence in the sense of Definition~\ref{defnew10}.\\
$\operatorname{ (2)}$ From Definition~\ref{partemer}, it follows that
the functor $V_{\mathcal{A}\mathcal{B}}$ is a homomorphism
between $\mathcal{A}$ and $\mathcal{B}$; however, it not a strong homomorphism.
\end{remark}

\begin{proposition}\label{partemeprop}
Let ${\mathcal{E}}^{p}_{\mathcal{A}\mathcal{B}}$ and
${\mathcal{E}}^{p}_{\mathcal{B}\mathcal{C}}$ be partial emergences. Then
${\mathcal{E}}^{p}_{\mathcal{A}\mathcal{C}}$ is also a partial emergence.
\end{proposition}
\begin{proof}
Let $V_{\mathcal{A}\mathcal{B}}:\mathcal{A}\longrightarrow \mathcal{B}$ and
$V_{\mathcal{B}\mathcal{C}}:\mathcal{B}\longrightarrow \mathcal{C}$ be GU partial functors.
We know that $V_{\mathcal{B}\mathcal{C}} \circ V_{\mathcal{A}\mathcal{B}}:
\mathcal{A}\longrightarrow \mathcal{C}$ is a functor from $\mathcal{A}$ to $\mathcal{C}$.
It is clear that $|e_{\mathcal{A}}| > |e_{\mathcal{C}}|$. Moreover, one has
$U_{\mathcal{C}}\circ V_{\mathcal{B}\mathcal{C}}=U_{\mathcal{B}}$ and
$U_{\mathcal{B}}\circ V_{\mathcal{A}\mathcal{B}}=U_{\mathcal{A}}$; so
$U_{\mathcal{C}}\circ (V_{\mathcal{B}\mathcal{C}} \circ
V_{\mathcal{A}\mathcal{B}})= U_{\mathcal{A}}$, that is, the diagram
\begin{eqnarray*}
\begin{tikzpicture}
  \matrix (m) [matrix of math nodes,row sep=3em,column sep=3em,minimum width=2em]
  {
     \mathcal{C} &  \\
      \mathcal{B} &  \mathcal{S} \\
     \mathcal{A}  & \\};
  \path[-stealth]
    (m-1-1) edge node [above] {$U_{\mathcal{C}}$} (m-2-2)
    (m-2-1) edge node [above] {$U_{\mathcal{B}}$} (m-2-2)
    (m-3-1) edge node [below] {$U_{\mathcal{A}}$} (m-2-2)
    (m-3-1) edge node [left] {$V_{\mathcal{A}\mathcal{B}}$}(m-2-1)
    (m-2-1) edge node [left] {$V_{\mathcal{B}\mathcal{C}}$} (m-1-1);
    \end{tikzpicture}
\end{eqnarray*}
commutes. Therefore, ${\mathcal{E}}^{p}_{\mathcal{A}\mathcal{C}}$ is a partial
emergence, as required.
\end{proof}

\begin{remark}\label{partemebio}
$\operatorname{(1)}$ It is interesting to note that Proposition~\ref{partemeprop} is
natural in the context of biological systems. More
precisely, it is natural that, if there exists an emergence phenomenon between
two stages $A$ and $B$ of a given organism and if there is
another emergence phenomenon between two stages $B$ and $C$ of the
same organism, then it is necessary the existence of an emergence
phenomenon between the stages $A$ and $C$ of such organism.
This fact is a good indicative that our theory is
very correlated with real phenomena with respect to ``living things".\\
$\operatorname{(2)}$ Note that partial emergences are not reflexive
nor symmetric. This is also natural in the scenario of biological
systems, in the sense that if a given system $\mathcal{A}$ is
``more emergent" than the system $\mathcal{B}$, then $\mathcal{B}$
cannot be ``more emergent" than $\mathcal{A}$. Furthermore, a given
system $\mathcal{A}$ cannot be ``more emergent" than itself.
\end{remark}

\begin{corollary}\label{corparteme}
For each $i = 1, 2, \ldots , n-1$, where $n > 1$ is an integer, assume that
${\mathcal{E}}^{p}_{{\mathcal{A}}_{i}{\mathcal{A}}_{i+1}}$ are partial
emergences. Then
${\mathcal{E}}^{p}_{{\mathcal{A}}_{1}{\mathcal{A}}_{n}}$ is also a partial emergence.
\end{corollary}
\begin{proof}
Follows directly by induction on $n$ and from Proposition~\ref{partemeprop}.
\end{proof}

Another important concept to be introduced is the \emph{relative emergence}.
As one can see in the next definition, such type of emergence does not
take into account the respectively underlying functors of the corresponding emergences.
The idea behind of this concept is to compare uniquely the constructs and
their respective number of operations.

\begin{definition}\label{relaemer}
Let $\mathcal{A}$ and $\mathcal{B}$ be constructs. A relative emergence
${\mathcal{E}}^{r}_{\mathcal{A}\rightarrow\mathcal{B}}$ is an ordered
triple $ {\mathcal{E}}^{r}_{\mathcal{A}\rightarrow\mathcal{B}}= ( \mathcal{A},
\mathcal{B}, R_{\mathcal{A}\mathcal{B}})$,
where $R_{\mathcal{A}\mathcal{B}}:\mathcal{A}\longrightarrow \mathcal{B}$ is a functor
and $|e_{\mathcal{A}}| > |e_{\mathcal{B}}|$.
The functor $R_{\mathcal{A}\mathcal{B}}$ is
called GU relative functor. The degree of the relative emergence
$\partial[{\mathcal{E}}^{r}_{\mathcal{A}\rightarrow\mathcal{B}}]$ is also defined as
$\partial[{\mathcal{E}}^{r}_{\mathcal{A}\rightarrow\mathcal{B}}] =
|e_{\mathcal{A}}| - |e_{\mathcal{B}}|$.
\end{definition}

\begin{proposition}\label{relaemeprop}
Let ${\mathcal{E}}^{r}_{\mathcal{A}\rightarrow\mathcal{B}}$ and
${\mathcal{E}}^{r}_{\mathcal{B}\rightarrow\mathcal{C}}$ be
relative emergences. Then ${\mathcal{E}}^{r}_{\mathcal{A}\rightarrow\mathcal{C}}$
is also a relative emergence.
\end{proposition}
\begin{proof}
Follows directly.
\end{proof}

From Proposition~\ref{relaemeprop}, it follows that a relative emergence
is transitive. However, it is not reflexive nor symmetric. Note that
the same discussion presented in Remark~\ref{partemebio} can also be considered
here, in the case of relative emergence.

\begin{corollary}\label{corelaeme}
For each $i = 1, 2, \ldots , n-1$, where $n > 1$ is an integer,
assume that ${\mathcal{E}}^{r}_{{\mathcal{A}}_{i} \rightarrow{\mathcal{A}}_{i+1}}$
are relative emergences. Then ${\mathcal{E}}^{r}_{{\mathcal{A}}_{1}\rightarrow{\mathcal{A}}_{n}}$
is also a relative emergence.
\end{corollary}
\begin{proof}
Follows directly from Proposition~\ref{relaemeprop} and by induction on $n$.
\end{proof}

Here, we define \emph{initial emergence} in a natural way.

\begin{definition}\label{iniquasi}
An emergence $ {\mathcal{E}}_{\mathcal{A}}= ( \mathcal{A},
e_{\mathcal{A}}, U_{\mathcal{A}} )$ is said to be initial if
for each emergence $ {\mathcal{E}}_{\mathcal{B}}=( \mathcal{B}, e_{\mathcal{B}},
U_{\mathcal{B}} )$, there exists a unique homomorphism
$T_{\mathcal{A}\mathcal{B}}$, $T_{\mathcal{A}\mathcal{B}}:
\mathcal{A} \longrightarrow \mathcal{B}$,
with $|e_{\mathcal{A}}|> |e_{\mathcal{B}}|$, that is,
$ {\mathcal{E}}^{r}_{\mathcal{A}\rightarrow\mathcal{B}}=( \mathcal{A},
\mathcal{B}, T_{\mathcal{A}\mathcal{B}})$ is a relative
emergence.
\end{definition}

\begin{example}
The empty category
${\mathcal{E}}_{\mathcal{C}(\emptyset)}=(\mathcal{C}(\emptyset),
e_{\mathcal{S}}=\emptyset,
U_{\mathcal{C}(\emptyset)}=Id_{\mathcal{C}(\emptyset)}=\emptyset)$, where
$\mathcal{C}(\emptyset)$ is the category without objects and without
morphisms is the unique initial category in the quasi-category whose
objects are all categories and the morphisms are all functors. Since the
empty category is not an emergence, there is no initial emergence.
\end{example}

\emph{Terminal emergences} can be also defined similarly.

\begin{definition}\label{termiquasi}
An emergence $ {\mathcal{E}}_{\mathcal{A}}= ( \mathcal{A},
e_{\mathcal{A}}, U_{\mathcal{A}} )$ is called terminal if
for each emergence $ {\mathcal{E}}_{\mathcal{B}}= ( \mathcal{B}, e_{\mathcal{B}},
U_{\mathcal{B}} )$ with $|e_{\mathcal{B}}|\geq 2$, there exists
a unique homomorphism
$T_{\mathcal{B}\mathcal{A}}: \mathcal{B} \longrightarrow \mathcal{A}$
such that ${\mathcal{E}}^{r}_{\mathcal{A}\rightarrow\mathcal{B}}=( \mathcal{A},
\mathcal{B}, T_{\mathcal{B}\mathcal{A}})$ is a relative emergence.
\end{definition}

\begin{example}
Let $\mathcal{B}$ be a construct with $|e_{\mathcal{B}}|=1$ and consider
the category $\mathcal{T}$ consisting of a unique
$\mathcal{B}$-object $B$ and the unique $\mathcal{B}$-morphism $B\xrightarrow{{id}_{B}}B$.
We form the emergence ${\mathcal{E}}_{\mathcal{T}}=
(\mathcal{T}, e_{\mathcal{T}}, U_{\mathcal{T}})$, where $U_{\mathcal{T}}$
is the usual underlying functor. For all emergence
$ {\mathcal{E}}_{\mathcal{A}}$ with $|e_{\mathcal{A}}|\geq 2$, there
exists a unique functor $F_{\mathcal{A}\mathcal{T}}: \mathcal{A}
\longrightarrow \mathcal{T}$. The functor sends all
$\mathcal{A}$-objects into $B$ and all $\mathcal{A}$-morphisms into ${id}_{B}$.
Since ${\mathcal{E}}^{r}_{\mathcal{A}\rightarrow\mathcal{T}}=( \mathcal{A},
\mathcal{T}, F_{\mathcal{A}\mathcal{T}})$ is a relative emergence, it
follows that $\mathcal{T}$ is terminal.
\end{example}

\begin{remark}\label{zeroemer}
The emergence $ {\mathcal{E}}_{\mathcal{A}}= ( \mathcal{A},
e_{\mathcal{A}}, U_{\mathcal{A}} )$ is called zero emergence
if it is both initial and terminal emergence. Since there is no initial
emergence, zero emergences do not exist.
\end{remark}

\subsection{Semi-emergence}\label{subsemieme}

In order to introduce \emph{semi-emergence}, we must define the
concept of \emph{generalized semi-underlying functor}.
In order to do this, we need to introduce some notation. Given a functor
$T:\mathcal{A}\longrightarrow \mathcal{B}$, we write
${\operatorname{Im}}_{\operatorname{(Ob)}}[T]$ to
denote the image of $T$ only in the objects of $\mathcal{A}$.

\begin{definition}\label{defnew9semi}
Let $\mathcal{A}$ be a construct. A generalized semi-underlying (GSU)
functor $U^{s}_{\mathcal{A}}$ with domain $\mathcal{A}$ is a functor
$U^{s}_{\mathcal{A}}: \mathcal{A}\longrightarrow \mathcal{S}$, $U^{s}_{\mathcal{A}}
(A\xrightarrow{f} A^{*})= \underline{B}
\xrightarrow{U^{s}_{\mathcal{A}}(f)} {\underline{B}}^{*}$, such that
$U^{s}_{\mathcal{A}}$ satisfies
${\operatorname{Im}}_{\operatorname{(Ob)}}[U^{s}_{\mathcal{A}}]=
{\operatorname{Im}}_{\operatorname{(Ob)}}[U_{\mathcal{A}}]$ and
$U^{s}_{\mathcal{A}}(f)$ is any function from
$\underline{B}$ to ${\underline{B}}^{*}$.
\end{definition}

\begin{remark}
It is interesting to note that in the previous definition we can
have $\underline{B}\neq\underline{A}$ or ${\underline{B}}^{*}
\neq{\underline{A}}^{*}$, and consequently, $U^{s}_{\mathcal{A}}(f)$ can also be
different from $\underline{f}$. The unique condition that
Definition~\ref{defnew9semi} imposes is that the image
(in the objects of $\mathcal{A}$) of both
$U^{s}_{\mathcal{A}}$ and $U_{\mathcal{A}}$ are equal.
\end{remark}

Based on the concept of SGU functor, we can now introduce semi-emergence.

\begin{definition}\label{defnew10semiequi}
A semi-emergence is an ordered triple ${\mathcal{E}}^{s}_{\mathcal{A}}=
( \mathcal{A}, e_{\mathcal{A}}, U^{s}_{\mathcal{A}})$,
where $\mathcal{A}$ is a construct, $e_{\mathcal{A}}$ is a finite set of operations
and $U^{s}_{\mathcal{A}}$ is a SGU functor
$U^{s}_{\mathcal{A}}:\mathcal{A}\longrightarrow \mathcal{S}$. The
order of $ {\mathcal{E}}_{\mathcal{A}}$ is defined as
$|e_{\mathcal{A}}|$.
\end{definition}

\begin{remark}
It is extremely important to note that the definition of semi-emergence
is almost analogous to that of emergence. However, there is
an important difference between them: in the former case, a given $\mathcal{A}$-object
$A$ is sent through $U^{s}_{\mathcal{A}}$ in some set of
${\operatorname{Im}}_{\operatorname{(Ob)}}[U_{\mathcal{A}}]$
and not necessarily in its underlying set
$\underline{A}$. This fact must be interpreted as
``keep the physical material the same" although the properties are lost,
as occurs concretely in emergent phenomena.
\end{remark}

Semi-emergences can be correlated by means of \emph{semi-homomorphisms}.

\begin{definition}\label{defnew14Asemi}
Let ${\mathcal{E}}^{s}_{\mathcal{A}}= ( \mathcal{A}, e_{\mathcal{A}},
U^{s}_{\mathcal{A}} )$ and $ {\mathcal{E}}^{s}_{\mathcal{B}}=
( \mathcal{B}, e_{\mathcal{B}}, U^{s}_{\mathcal{B}} )$ be semi-emergences. We
say that ${\mathcal{E}}^{s}_{\mathcal{A}}$ is semi-homomorphic
to ${\mathcal{E}}^{s}_{\mathcal{B}}$ if there exists a functor $F:\mathcal{A}
\longrightarrow\mathcal{B}$, called semi-homomorphism, such that
$U^{s}_{\mathcal{B}} \circ F = U^{s}_{\mathcal{A}}$.
We write $ {\mathcal{E}}^{s}_{\mathcal{A}} {\simeq}_{s}
{\mathcal{E}}^{s}_{\mathcal{B}}$ to denote that $ {\mathcal{E}}^{s}_{\mathcal{A}}$
is semi-homomorphic to ${\mathcal{E}}^{s}_{\mathcal{B}}$.
\end{definition}

\begin{proposition}\label{semihomouni}
Assume that the semi-emergence ${\mathcal{E}}^{s}_{\mathcal{A}}=
( \mathcal{A}, e_{\mathcal{A}}, U^{s}_{\mathcal{A}} )$ is
semi-homomorphic to ${\mathcal{E}}^{s}_{\mathcal{B}}= ( \mathcal{B}, e_{\mathcal{B}},
U^{s}_{\mathcal{B}} )$ by means of the semi-homomorphism $F$.
If $U^{s}_{\mathcal{B}}$ is faithful, then $F$ is uniquely determined by
its values on objects.
\end{proposition}
\begin{proof}
Let $F, G:\mathcal{A} \longrightarrow\mathcal{B}$ be semi-homomorphisms
such that, for each $\mathcal{A}$-object $A$, $F(A)=G(A)$. Thus, for each
$\mathcal{A}$-morphism $A \xrightarrow{f} A^{*}$ we have $F(A)=G(A)\xrightarrow{F(f)}
F(A^{*})=G(A^{*})$ and $F(A)=G(A)\xrightarrow{G(f)}F(A^{*})=G(A^{*})$, where
$[U^{s}_{\mathcal{B}} \circ F](f) = U^{s}_{\mathcal{A}}(f)=[U^{s}_{\mathcal{B}}
\circ G](f)$. Since $U^{s}_{\mathcal{B}}$ is faithful, it follows that
$F(f) = G(f)$ for all $\mathcal{A}$-morphism; so, $F=G$.
\end{proof}

Similarly, semi-equivalence of semi-emergences can be also introduced.

\begin{definition}\label{defequivsemiemer}
Let ${\mathcal{E}}^{s}_{\mathcal{A}}= ( \mathcal{A}, e_{\mathcal{A}},
U^{s}_{\mathcal{A}} )$ and ${\mathcal{E}}^{s}_{\mathcal{B}}=
( \mathcal{B}, e_{\mathcal{B}}, U^{s}_{\mathcal{B}} )$ be semi-emergences.
A semi-homomorphism $F:\mathcal{A} \longrightarrow \mathcal{B}$ is said to
be a semi-equivalence from ${\mathcal{E}}_{\mathcal{A}}$ to
${\mathcal{E}}_{\mathcal{B}}$ if $F$ is full, faithful and
isomorphism-dense. If there exists a semi-equivalence from
$\mathcal{A}$ to $\mathcal{B}$ we say that
${\mathcal{E}}_{\mathcal{A}}$ is semi-equivalent to
${\mathcal{E}}_{\mathcal{B}}$.
\end{definition}

In the following, we introduce the concept o semi-isomorphism
between semi-emergences.

\begin{definition}\label{defnew14Asemiiso}
Let ${\mathcal{E}}^{s}_{\mathcal{A}}= ( \mathcal{A}, e_{\mathcal{A}},
U^{s}_{\mathcal{A}} )$ and $ {\mathcal{E}}^{s}_{\mathcal{B}}=
( \mathcal{B}, e_{\mathcal{B}}, U^{s}_{\mathcal{B}} )$ be semi-emergences. We
say that ${\mathcal{E}}^{s}_{\mathcal{A}}$ is semi-isomorphic
to ${\mathcal{E}}^{s}_{\mathcal{B}}$ if there exists a functor $F:\mathcal{A}
\longrightarrow\mathcal{B}$, called semi-isomorphism, such that
$\mathcal{A}$ and $\mathcal{B}$ are isomorphic as categories and
$U^{s}_{\mathcal{B}} \circ F = U^{s}_{\mathcal{A}}$.
We write $ {\mathcal{E}}^{s}_{\mathcal{A}} {\equiv}_{s}
{\mathcal{E}}^{s}_{\mathcal{B}}$ to denote that $ {\mathcal{E}}^{s}_{\mathcal{A}}$
is semi-isomorphic to ${\mathcal{E}}^{s}_{\mathcal{B}}$.
\end{definition}

\emph{Strong semi-homomorphisms} are introduced in the following way.

\begin{definition}\label{defnew14ASsemi}
Let $ {\mathcal{E}}^{s}_{\mathcal{A}}= ( \mathcal{A}, e_{\mathcal{A}},
U^{s}_{\mathcal{A}} )$ and $ {\mathcal{E}}^{s}_{\mathcal{B}}=
( \mathcal{B}, e_{\mathcal{B}}, U^{s}_{\mathcal{B}} )$ be semi-emergences.
We say that ${\mathcal{E}}^{s}_{\mathcal{A}}$ is strong
semi-homomorphic to ${\mathcal{E}}^{s}_{\mathcal{B}}$ if
there exists a functor $F:\mathcal{A} \longrightarrow\mathcal{B}$
(called strong semi-homomorphism) such that $U^{s}_{\mathcal{B}}
\circ F = U^{s}_{\mathcal{A}}$ and $|e_{\mathcal{A}}|= |e_{\mathcal{B}}|$.
We write $ {\mathcal{E}}^{s}_{\mathcal{A}} {\simeq}_{s}^{st}
{\mathcal{E}}_{\mathcal{B}}$ to denote that
${\mathcal{E}}^{s}_{\mathcal{A}}$ is strong semi-homomorphic
to ${\mathcal{E}}^{s}_{\mathcal{B}}$.
\end{definition}

Here, we define what we mean by \emph{strong semi-isomorphism}.

\begin{definition}\label{defnew14ASsemiisostrong}
Let $ {\mathcal{E}}^{s}_{\mathcal{A}}= ( \mathcal{A}, e_{\mathcal{A}},
U^{s}_{\mathcal{A}} )$ and $ {\mathcal{E}}^{s}_{\mathcal{B}}=
( \mathcal{B}, e_{\mathcal{B}}, U^{s}_{\mathcal{B}} )$ be semi-emergences.
We say that ${\mathcal{E}}^{s}_{\mathcal{A}}$ is strong
semi-isomorphic to ${\mathcal{E}}^{s}_{\mathcal{B}}$ if
there exists a functor $F:\mathcal{A} \longrightarrow\mathcal{B}$
(called strong semi-isomorphism) such that $\mathcal{A}$ and
$\mathcal{B}$ are isomorphic as categories, $U^{s}_{\mathcal{B}}
\circ F = U^{s}_{\mathcal{A}}$ and $|e_{\mathcal{A}}|= |e_{\mathcal{B}}|$.
We write $ {\mathcal{E}}^{s}_{\mathcal{A}} {\equiv}_{s}^{st}
{\mathcal{E}}_{\mathcal{B}}$ to denote that
${\mathcal{E}}^{s}_{\mathcal{A}}$ is strong semi-isomorphic
to ${\mathcal{E}}^{s}_{\mathcal{B}}$.
\end{definition}

Note that our definitions of semi-emergences and their corresponding
(strong) semi-homomorphisms generalize that of emergence
and their corresponding homomorphisms. Additionally, such concepts preserve
nice properties, as establishes the following result.

\begin{proposition}\label{propstrongsemiisohomo}
The following conditions hold:
\begin{itemize}
\item [ $\operatorname{(1)}$] The relations ${\simeq}_{s}$ and
${\simeq}_{s}^{st}$ are both reflexive
and transitive.

\item [ $\operatorname{(2)}$] The relations ${\equiv}_{s}$ and
${\equiv}_{s}^{st}$ are both equivalence relations.
\end{itemize}
\end{proposition}
\begin{proof}
The proofs are similar to that of Propositions~\ref{prophomom}~and~\ref{newteo15},
respectively.
\end{proof}

\begin{remark}
The concepts of product (co-products), equalizer (co-equalizer),
pullback (pushout), limits (co-limits) of semi-emergences
and strong semi-emer-gences can be also defined similarly to these concepts
with respect to emergences. However, because the ideas are similar,
and to avoid that the length of the paper becomes excessive long,
we do not present such concepts here.
\end{remark}

\subsection{Internal structures in emergence}

In the previous sections we introduced concepts and definitions on emergence
and, with such theory, we have shown results and correlations among them.
Here, we are concerned to explore the operations of a given emergence, such as
existence of products, limits and so on. Because an emergence is a triple
$ {\mathcal{E}}_{\mathcal{A}}= ( \mathcal{A}, e_{\mathcal{A}},
U_{\mathcal{A}} )$ where $\mathcal{A}$ is a construct, we can
borrow its structure to define similar concept in the emergence context. For
example, we can easily define the concept a given emergence to have products:

\begin{definition}\label{interprodu}
Let $ {\mathcal{E}}_{\mathcal{A}}= ( \mathcal{A}, e_{\mathcal{A}},
U_{\mathcal{A}} )$ be an emergence. We say that
${\mathcal{E}}_{\mathcal{A}}$ has products if the construct $\mathcal{A}$
has product. In other words, for every family ${(A_{i})}_{i \in I}$
of $\mathcal{A}$-objects indexed for a set $I$ there exists a product
${({\prod}A_{i}\xrightarrow{{\pi}_{j}}A_{j})}_{I}$ in the sense
of Definition~10.19 in \cite{strecker:1990}. Analogously, the
emergence has finite products if every finite family
${(A_{i})}_{i \in I}$ of $\mathcal{A}$-objects there exists a product
${({\prod}A_{i}\xrightarrow{{\pi}_{j}}A_{j})}_{I}$ (according to
Definition~10.19 in \cite{strecker:1990}).
\end{definition}

As an example, let us consider the emergence $(\mathcal{F},  e_{\mathcal{F}},
U^{s}_{\mathcal{F}} )$, where $\mathcal{F}$ is the construct of all fields
and $U^{s}_{\mathcal{F}}$ is the usual underlying functor. Then it is known that
$\mathcal{F}$ does not have finite products. On the other hand, if one
considers the emergence $(\mathcal{V},  e_{\mathcal{V}},
U^{s}_{\mathcal{V}} )$ of real vectors spaces, where $U^{s}_{\mathcal{V}}$
is the usual underlying functor, then such emergence has products.

We can also rephrase Proposition~10.30 in \cite{strecker:1990} in order to have
conditions in which a given emergence has products:

\begin{proposition}\label{condintemer}
Let $ {\mathcal{E}}_{\mathcal{A}}= ( \mathcal{A}, e_{\mathcal{A}},
U_{\mathcal{A}} )$ be an emergence. Then
$ {\mathcal{E}}_{\mathcal{A}}$ has finite products if and only if $\mathcal{A}$
has terminal objects and products of pairs of objects.
\end{proposition}
\begin{proof}
See the proof of Proposition~10.30 in \cite{strecker:1990}.
\end{proof}

Recall that a \emph{lattice} is a partially ordered set in which every
pair of elements has a meet and a join. The class of all lattices (objects)
together with all lattice homomorphisms (that is, all maps preserving meets and
joins of pairs) as morphisms form a category under the usual composite
of functions and usual identities. A lattice is said to be complete if every
subset has both meet and join. Rephrasing Theorem~10.32 in
\cite{strecker:1990}, we have the following result.

\begin{theorem}\label{prodemergexist}
$\operatorname{(1)}$ An emergence $ {\mathcal{E}}_{\mathcal{A}}
= ( \mathcal{A}, e_{\mathcal{A}}, U_{\mathcal{A}} )$
that has products for all class-indexed families must be thin.\\
$\operatorname{(2)}$ A small emergence has products if and
only if it is equivalent to a complete lattice.
\end{theorem}
\begin{proof}
See the proof of Theorem~10.32 in \cite{strecker:1990}.
\end{proof}

Recall that a category $\mathcal{A}$ is said to have intersections if
for every $\mathcal{A}$-object $A$ e for every family of sub-objects
of $A$, there exists an intersection. We then can rewrite \cite[Definition 12.2]{strecker:1990}
in the context of emergences in a natural way.

\begin{definition}
An emergence $ {\mathcal{E}}_{\mathcal{A}}$ is said to be
finitely complete if for every
finite diagram in $\mathcal{A}$ there exists a limit; complete, if
for each small diagram in $\mathcal{A}$ there exists a limit; strong complete
if it is complete and has intersections.
\end{definition}

Theorem 12.3 of \cite{strecker:1990} reads as follows in terms of emergence.

\begin{theorem}
For every emergence ${\mathcal{E}}_{\mathcal{A}}$ the conditions
are equivalent:\\
$\operatorname{(1)}$ ${\mathcal{E}}_{\mathcal{A}}$ is complete;\\
$\operatorname{(2)}$ ${\mathcal{E}}_{\mathcal{A}}$ has products
and equalizers;\\
$\operatorname{(3)}$ ${\mathcal{E}}_{\mathcal{A}}$ has products
and finite intersections.
\end{theorem}
\begin{proof}
See the proof of Theorem 12.3 in \cite{strecker:1990}.
\end{proof}

Since the concepts of co-product, co-limits, co-equalizer are dual
of products, limits and equalizer, respectively, we do not
consider none of these concepts here due to the well-known Duality Principle.
In fact, the concepts on Category Theory can be easily converted
into our emergence theory. Therefore, our theory encompasses a theory
in which emergence can be formally treated, specially emergence present
in biological systems.

\section{Abstracting emergence}\label{sec5}

This section has as main purpose to present a procedure in order
to close the gap between the formalization of emergence
presented in the previous section and a concrete biological system. This procedure
is the one presented by Robert Rosen in his works on $(M, R)$-systems
\cite{rosen:1958A,rosen:1958B,rosen:1959}. In the first place, Rosen showed
that a treatment of biological systems using graph theory is incomplete;
then he used categories for this purpose. The resulting representation by
applying category is coined by him as abstract block diagram (ABD).

In order to give an idea for the reader about the results showed
by Rosen, we present here some theorems on ABD´s.

\begin{theorem}(Representation Theorem)\cite[Theorem 2]{rosen:1958B}
Given any system $M$ and a resolution of $M$ into components,
it is possible to find an abstract block diagram which
represents $M$ and which consists of a collection of suitable objects and
mappings from the category $\mathcal{S}$ of all sets.
\end{theorem}
According to Rosen, a resolution of a given system is, roughly speaking,
to decompose it into components.

\begin{theorem}(Canonical Form Theorem)\cite[Theorem 3]{rosen:1958B}
Given a block diagram for an arbitrary system $M$, we can find an
abstract block diagram represented $M$ such that none of the
mappings of the abstract diagram is factorable through any sub-product
of its domain.
\end{theorem}

\begin{theorem}\cite[Theorem 4]{rosen:1958B}
Given a decomposition of an arbitrary system $M$ into components,
we we may find a further decomposition of $M$ such
that every component of the new decomposition emits exactly one output.
\end{theorem}

The main result proposed by Rosen is given in the following.
\begin{theorem}\cite[Theorem 6]{rosen:1958B}
Let $M$ be an abstract block diagram which represents a definite
biological system. Let $T$ be a faithful functor. Then $T(M)$
is an abstract block diagram which represents the system if and only if
$T$ is regular and multiplicative.
\end{theorem}
In this context, $T: \mathcal{S}\longrightarrow\mathcal{S}$ is said
to be regular if it satisfies:$(1)$ If $A \in
\operatorname{Ob}(\mathcal{S})$ and $A\neq\emptyset $, then $T(A)\neq\emptyset$;
$(2)$ If $A\subset B$, then $T(A)\subset T(B)$. The functor
$T$ is called multiplicative if for any two sets $A_{1} , A_{2}$, one
has $T(A_{1})\times T(A_{2})= T(A_{1}\times A_{2})$.

To generate an ABD it is necessary to access empirically a given
biological system, such as: cells, tissues, organs,
organic systems, or the individual. In any case, it is possible to decompose
such system into parts and their relationships. The parts of
the systems, which are considered the biological agents of the system,
are associated with the \emph{hom} sets, while the relationships between agents
are considered as objects. It is important to note that each object is considered as
an output or input of a given \emph{hom} set. To clarify such concept, let us consider
a ${hom}_1$ that generates an output $A$ that serves as input to a second
${hom}_{2}$. This means that $A$ belongs to the codomain of all morphisms
in ${hom}_{1}$, while it also belongs to the domain of the ${hom}_{2}$.
It is clear from the context that each morphism is a mapping between inputs
and outputs of a \emph{hom} set. This is the modeling of the operations made
by a biological agent, like a cell in a tissue, or a organelle in a cell,
a tissue in a organ, and so on. The operations of the objects stands
for its own properties that are given by its kind. In a tissue context system,
a cell can be an input or and output, at the same time that it has
its own operations. The set of all operations of objects in an ABD
is the internal structure of the system. With this concepts in mind, one
can decompose a biological system into parts, after associating them with the \emph{hom}
sets regarding their mappings; we then can consider the outputs and inputs
between \emph{hom} sets as structured sets. In particular, the $\mathcal{S}$
category is a case in which its objects does not have an internal structure.
Note that the empirical criteria to describe such category is given by a minimal
agency of inputs and outputs, such as atoms or molecules. However, the level in
which the $\mathcal{S}$ category should be considered is dependent in the
empirical rigor that a specific experimental setting require. In most cases, we use
the basic level of atoms associating them to the $\mathcal{S}$ category, and the
other constructs are utilized to represent other levels of abstraction.
In this form, it is possible to induce an emergence based on an ABD built.

\section{Biological implications}\label{sec6}

In this section, we discuss the proposed mathematical theory and we
explain how it fits in studies involving biological systems. In the
previous section, we suggested the abstraction of emergence from the decomposition
of biological systems. For each concrete biological system, there exists
many forms to perform this decomposition.

Note that any of these decompositions represents an ABD of the same concrete
system. For instance, let us consider a simple organism which is
multicellular, has simple tissues, organs, and organic systems, like a flat worm.
We can decompose such organism by its organic systems, tissues, organs, cells,
molecules, and atoms; each of these these cases stands for the same phenomenon.
However, the properties and features within these cases are much different.
This is true since the biological agents, which constitutes the parts of the
whole, differs in nature. In other words, a set of interacting cells does not
represent a tissue, meanwhile a set of tissues does not represent an organ, and so on.
This happens because a higher decomposition level of hierarchy is
emergent in terms of any other lower level. This fact implies that each
decomposition of the same concrete biological system induces an hierarchy based
on how emergent each level is among themselves.

Connecting Relational Biology and our proposed theory, we define
emergence as an ordered triple $ {\mathcal{E}}_{\mathcal{A}}=
( \mathcal{A}, e_{\mathcal{A}}, U_{\mathcal{A}})$ (see Definition~\ref{defnew10}),
in which the construct $\mathcal{A}$ is an
ABD generated by decomposition. More specifically,
$\operatorname{Ob}(\mathcal{A})$ are inputs and outputs,
$\operatorname{Mor}(\mathcal{A})$ are all operations of the system
and the \emph{hom} sets are the biological agents that compose the system. Note
that if $\mathcal{A}=\mathcal{S}$, then there is no emergence since there is
no internal structure. In this context, an emergent phenomenon is interpreted
as the existence of internal structures in the objects
$\operatorname{Ob}(\mathcal{A})$ and the absence of
any structure in their underlying class of objects $\operatorname{Ob}(\mathcal{S})$.
Note that when considered as sets, the objects of $\operatorname{Ob}(\mathcal{A})$ and
$\operatorname{Ob}(\mathcal{S})$ are the same, but as objects they are quite different,
since the objects of $\mathcal{A}$ have internal structure and the objects of
$\mathcal{S}$ do not have. To make our assumption more clear,
we take the category $\operatorname{Ob}(\mathcal{S})$
as the ABD of maximal decomposition, say, atomic – which corresponds to the system
in which its parts bear minimal functionality in biological terms.
In this basal category, all emergent effects on its parts are lost due to reduction.
We emphasize that such effect is present in both inputs and outputs
conceived as Cartesian product of objects, rather than in its agents
conceived as \emph{hom} sets. Since the organizational level increases,
new laws and effects takes place, and
more emergent the system becomes. This is true because a lower level cannot
explain the higher ones. At the same time, from any higher level of
organization, it is possible to explain and understand any lower level.

Returning to the example of the flat worm, when the worm is divided into cells,
there will be ``less information" about the life of the worm when
comparing with the worm divided in tissues. This fact occurs because all tissues are
composed by cells and their corresponding
operations; hence, we can access all information from tissue to cell, but the
reciprocal is not true. The same interpretation can be made to any other pair
of representation via ABDs. It is important to note that the emergent
criteria is not in the physical
realization of the biological agents (which corresponds to morphisms) nor in its
material composition (which corresponds to their underlying sets).
It lies, in fact, in the operations within inputs and outputs and in its
algebraical functionality that acts among \emph{hom} sets.

\section{Final Remarks}\label{sec7}

In this paper, we have proposed a theory to describe the concept of emergence.
As it is well-known, biological phenomena are, in fact, emergent
phenomena. Since they are hard to be investigated, we have chosen a powerful mathematical
tool to do this task, namely, Theory of Categories. 
We utilized constructs, their operations and
their corresponding generalized underlying functor to characterize emergence. 
After this, we introduced and showed several results concerning 
homomorphism (isomorphism) between 
emergences, representability, pullback, pushout, equalizer, product and 
co-product of emergences among other concepts. Our theory can be abstracted 
from a concrete biological phenomenon by means of ABDs.
This work is the beginning
of our investigation concerning emergence
phenomena and how they can be explored by means of categories.
As future works, there are much research to be done and much
correlation to be found. In this light we intend, in a near future, to expand
and develop the presented theory in order to
improve the interpretation of some new situations of the real life.

\begin{center}
\textbf{Acknowledgements}
\end{center}
This research has been partially supported by the Brazilian Agencies
CAPES and CNPq.

\end{document}